\documentclass[11pt,a4paper]{article}
\usepackage{euscript}
\usepackage{amsmath}
\usepackage{graphicx}
\usepackage{amsthm}
\usepackage{amssymb}
\usepackage{epsfig}
\usepackage[all]{xy}
\usepackage{url}
\usepackage{framed}
\newtheorem{theorem}{Theorem}[section]
\newtheorem{corollary}[theorem]{Corollary}

\newtheorem{lemma}[theorem]{Lemma}
\newtheorem{proposition}[theorem]{Proposition}

\newtheorem{definition}{Definition}[section]

\newtheorem{example}{Example}[section]
\newtheorem{remark}{Remark}[section]
\numberwithin{equation}{section}

\title{F-schemes}
\date{}
\author{Masood Aryapoor \footnote{masood.aryapoor@mail.ipm.ir}
\footnote{This research was in part supported by a grant from IPM.}
\\ \\
School of Mathematics, \\
Institute for Research in Fundamental Sciences (IPM),\\
P.O.Box: 19395-5746, Tehran, Iran.
}
\begin{document}
\maketitle
\begin{abstract}
\noindent
In this paper, the notion of F-schemes, a ``generalization'' of schemes,
is introduced to cover
unitary noncommutative rings.
\end{abstract}
\begin{section}{Introduction}
The modern theory of (commutative) algebraic geometry began
with the work of Alexander Grothendieck.
In this theory, the main idea is to associate a locally ringed space
$(\text{Spec}(A),\mathcal{O}_{\text{Spec}(A)})$, called
an affine scheme
(which is a geometric object), to any commutative ring $A$.  This provides a
powerful method  to study commutative rings
using geometric ideas. So it is very tempting to apply the same approach to study
noncommutative rings.
There have been attempts to find a theory of ``noncommutative algebraic geometry'',
see \cite{Kap,KR,Ro,SV,Va,VV}
for example.
However, the construction of ``affine noncommutative schemes'' for general
noncommutative rings  has not been quite successful.

To construct an affine scheme, there are two main ingredients, namely ``(prime) ideals'' and
``localization''. One of the
main obstacles in constructing ``affine noncommutative schemes''  is the lack of
a relation between localization and (prime) ideals in the noncommutative case.
For example, a simple ring has no nontrivial ideals whereas
it might be possible to find a nontrivial localization of it. So it seems that in order
to construct such an affine noncommutative scheme,
one should discard one of these ingredients. Keeping localization, one is led
to the notion of F-schemes.

The idea behind the construction of F-schemes is as follows: recall that an affine
variety is constructed from the set of
maximal ideals of certain
commutative rings. In order to construct an affine scheme, one must extend the set of
maximal ideals to the set of prime ideals. Then affine schemes include all commutative rings.
In the same way, by extending the set of prime ideals to a bigger set
(namely the set of fully invertible systems defined
in this paper), one constructs  affine F-schemes which will include all unitary noncommutative rings.
Affine F-schemes are certain (not necessarily commutative or locally) ringed spaces and
hence can be glued together to obtain some other spaces called ``F-schemes''.

As mentioned earlier, F-schemes are less geometrical
in nature than schemes. For example, the stalks of their structure sheaf
are not necessarily local rings. This makes it harder to apply geometrical
ideas to F-schemes. Nevertheless, some of the familiar notions and constructions in the
theory of schemes can be generalized to F-schemes.  Among these are the notions of
``closed subscheme, separated schemes and projective schemes'' which are considered in this paper.
These natural generalizations  make one believe 
that it is worth studying F-schemes.

There are connections between this work and some other works in ``noncommutative
algebraic geometry''. However there is little explanation on these connections in this paper
which I hope to discuss somewhere else. 

Here is a few words about the structure of this paper. In writing this paper, I
have  more or less followed the style of Hartshorne's book (\cite{Ha}) on algebraic geometry.
This paper has five parts. 
Part one includes the background materials.   
In part two, F-schemes are introduced. In part three, the relation between schemes
and F-schemes is discussed. The notions of closed sub-F-schemes and
separated F-schemes are introduced in part four. Finally, in part five,
the construction of projective F-schemes is discussed.

\textbf{Acknowledgement}: I would like to thank M. M. Kapranov for
pointing out the reference
\cite{Va}.

\end{section}

%%%%%%%%%%%%%%%%%%%%%%%%%%%%%%%%%%%%%%%%%%%%%%%%%%%%%%%%%%%%%%%%%%%%%%%%%%%%%%%%%%%%%%%%%%%%%%%%
%%%%%%%%%%%%%%%%%%%%%%%%%%%%%%%%%%%%%%%%%%%%%%%%%%%%%%%%%%%%%%%%%%%%%%%%%%%%%%%%%%%%%%%%%%%%%%%%
%%%%%%%%%%%%%%%%%%%%%%%%%%%%%%%%%%%%%%%%%%%%%%%%%%%%%%%%%%%%%%%%%%%%%%%%%%%%%%%%%%%%%%%%%%%%%%%%
%%%%%%%%%%%%%%%%%%%%%%%%%%%%%%%%%%%%%%%%%%%%%%%%%%%%%%%%%%%%%%%%%%%%%%%%%%%%%%%%%%%%%%%%%%%%%%%%
%%%%%%%%%%%%%%%%%%%%%%%%%%%%%%%%%%%%%%%%%%%%%%%%%%%%%%%%%%%%%%%%%%%%%%%%%%%%%%%%%%%%%%%%%%%%%%%%
%%%%%%%%%%%%%%%%%%%%%%%%%%%%%%%%%%%%%%%%%%%%%%%%%%%%%%%%%%%%%%%%%%%%%%%%%%%%%%%%%%%%%%%%%%%%%%%%
%%%%%%%%%%%%%%%%%%%%%%%%%%%%%%%%%%%%%%%%%%%%%%%%%%%%%%%%%%%%%%%%%%%%%%%%%%%%%%%%%%%%%%%%%%%%%%%%
%%%%%%%%%%%%%%%%%%%%%%%%%%%%%%%%%%%%%%%%%%%%%%%%%%%%%%%%%%%%%%%%%%%%%%%%%%%%%%%%%%%%%%%%%%%%%%%%
%%%%%%%%%%%%%%%%%%%%%%%%%%%%%%%%%%%%%%%%%%%%%%%%%%%%%%%%%%%%%%%%%%%%%%%%%%%%%%%%%%%%%%%%%%%%%%%%
%%%%%%%%%%%%%%%%%%%%%%%%%%%%%%%%%%%%%%%%%%%%%%%%%%%%%%%%%%%%%%%%%%%%%%%%%%%%%%%%%%%%%%%%%%%%%%%%
%%%%%%%%%%%%%%%%%%%%%%%%%%%%%%%%%%%%%%%%%%%%%%%%%%%%%%%%%%%%%%%%%%%%%%%%%%%%%%%%%%%%%%%%%%%%%%%%
%%%%%%%%%%%%%%%%%%%%%%%%%%%%%%%%%%%%%%%%%%%%%%%%%%%%%%%%%%%%%%%%%%%%%%%%%%%%%%%%%%%%%%%%%%%%%%%%
%%%%%%%%%%%%%%%%%%%%%%%%%%%%%%%%%%%%%%%%%%%%%%%%%%%%%%%%%%%%%%%%%%%%%%%%%%%%%%%%%%%%%%%%%%%%%%%%
%%%%%%%%%%%%%%%%%%%%%%%%%%%%%%%%%%%%%%%%%%%%%%%%%%%%%%%%%%%%%%%%%%%%%%%%%%%%%%%%%%%%%%%%%%%%%%%%
%%%%%%%%%%%%%%%%%%%%%%%%%%%%%%%%%%%%%%%%%%%%%%%%%%%%%%%%%%%%%%%%%%%%%%%%%%%%%%%%%%%%%%%%%%%%%%%%
%%%%%%%%%%%%%%%%%%%%%%%%%%%%%%%%%%%%%%%%%%%%%%%%%%%%%%%%%%%%%%%%%%%%%%%%%%%%%%%%%%%%%%%%%%%%%%%%

\begin{section}{Preliminaries}\label{localization S}

In this part the background materials required in defining F-schemes, are given.
Some knowledge of ring theory is assumed.

\noindent \textbf{Notations and conventions}: throughout this paper, all rings are assumed to be
unitary but not necessarily commutative.
For a ring $A$, the notation $\text{U}(A)$  stands for the group of
invertible elements in $A$.
Given rings $A$ and $B$,  a homomorphism $\phi:A\to B$ means a
ring homomorphism such that $\phi(1)=1$. Given a homomorphism $\phi:A\to B$ and
a subset $S\subset B$,  the notation $A\cap S$ is sometimes used
instead of $\phi^{-1}(S)$ if no confusion arises.

%%%%%%%%%%%%%%%%%%%%%%%%%%%%%%%%%%%%%%%%%%%%%%%%%%%%%%%%%%%%%%%%%%%%%%%%%%%%%%%%%%%%%%%%%%%%%%%%
%%%%%%%%%%%%%%%%%%%%%%%%%%%%%%%%%%%%%%%%%%%%%%%%%%%%%%%%%%%%%%%%%%%%%%%%%%%%%%%%%%%%%%%%%%%%%%%%
%%%%%%%%%%%%%%%%%%%%%%%%%%%%%%%%%%%%%%%%%%%%%%%%%%%%%%%%%%%%%%%%%%%%%%%%%%%%%%%%%%%%%%%%%%%%%%%%
%%%%%%%%%%%%%%%%%%%%%%%%%%%%%%%%%%%%%%%%%%%%%%%%%%%%%%%%%%%%%%%%%%%%%%%%%%%%%%%%%%%%%%%%%%%%%%%%

\begin{subsection}{Localization in ring theory}
One would like to ``invert'' elements in unitary rings. In general,
this process is not as
simple as in the case of commutative rings. Nevertheless, for any set
of elements in a ring,
there is a ring (which might be
the zero ring) in which any element of that set is invertible and has
the obvious universal property.
More precisely  (see \cite{Co})\\
\noindent \textbf{Localization rings}:
suppose that $A$ is a ring and $S\subset A$ is a subset of $A$. Then
there is a  ring $A_S$ and a
homomorphism $i_S:A\to A_S$ such that
\begin{enumerate}
\item for every $s\in S$, the element $i_S(s)$ is invertible in $A_S$ and
\item $A_S$ has the following universal property: given any
homomorphism $\phi:A\to B$ with $\phi(S)\subset\text{U}(B)$,
there is a unique homomorphism $\phi_S:A_S\to B$ such that $\phi=\phi_Si_S$,
i.e. the following
diagram is commutative
\begin{displaymath}
\xymatrix{ & A \ar[dr]_{\phi}  \ar[r]^{i_S} & A_S \ar[d]^{\phi_S} \\
           &   & B}
 \end{displaymath}
\end{enumerate}
The ring $A_S$ is called the \textbf{localization} of $A$ at $S$.
If $S=\{a\}$ ($a\in A$),
then the localization ring
  $A_S$ is simply denoted by $A_a$.
The ring $A_S$ can be presented as follows
$$A_S=\langle A,\{x_s\}_{s\in S}|\, x_ss=sx_s=1,\,\text{for any}\,
s\in S \rangle. $$
This is the presentation of $A_S$ that I will consider throughout this paper.
Note that elements of $A_S$ can be presented as $f(a_1,...,a_m,
x_{s_1},...,x_{s_n})$ where
$f(Y_1,...,Y_m,Z_1,...,Z_n)\in \mathbb{Z}\langle Y_1,...,Y_m,
Z_1,...,Z_n\rangle$ is a
noncommutative polynomial with
coefficients in $\mathbb{Z}$,
 $a_1,...,a_m \in A$, and $s_1,...,s_n\in S$.
\begin{remark}
It does not seem easy to describe the kernel of the   homomorphism
$i_S:A\to A_S$. However, it is clear that this kernel contains the ideal generated
by all elements $a$ of $A$ such that
$as=0$ or $sa=0$ for some $s\in S$.
\end{remark}
Regarding the localization rings $A_S$, one of the main issues is that $A_S$ might be
 the zero ring (for example if $0\in S$) which in this case, the localization is not
 interesting!
\begin{definition}
Suppose that $A$ is a   ring. A subset $S$ of $A$ is called \textbf{properly invertible}
if
there is a  homomorphism $\phi:A\to B$, from $A$  to a \underline{nonzero}  ring $B$,
such that  $\phi(S)\subset \text{U}(B)$.
\end{definition}
It  is easy to see that a subset $S\subset A$ is properly invertible iff $A_S\neq 0$.
 When  inverting elements in $S$,   some other elements of $A$ might also become invertible.
 If it is possible to only invert elements of $S$ then  $S$ is called fully invertible.
 More precisely
\begin{definition}
A  subset $S$ of $A$ is called \textbf{fully invertible} if  there is a
ring homomorphism $\phi:A\to B$, from $A$  to a nonzero ring $B$,
such that $ S=\phi^{-1}(\text{U}(B))$.
\end{definition}
In particular, taking  $\phi=\text{Id}:A\to A$ implies that $\text{U}(A)$ is a fully
invertible subset of $A$.
Clearly any fully invertible subset is properly invertible.
\begin{lemma}\label{fully invertiblesubsets L}
A  subset $S$ of $A$ is fully invertible iff  $S$ is properly invertible and
$S=A\cap \text{U}(A_S)$.
\end{lemma}
\begin{proof}
Let $S$ be fully invertible. Then there is a homomorphism $\phi:A\to B$ ($B\neq 0$),
 such that $S=\phi^{-1}(\text{U}(B))$. Therefore there is a  ring homomorphism
 $\phi_S:A_S\to B$
 such that $\phi=\phi_Si_S$.
So $S=\phi^{-1}(\text{U}(B))=i_S^{-1}(\phi_S^{-1}(\text{U}(B)))$ contains
$i_S^{-1}(\text{U}(A_S))$. Since
$S\subset i_S^{-1}(\text{U}(A_S))$, we have $S=i_S^{-1}(\text{U}(A_S))$.
The converse is clear.
\end{proof}
The following properties of
fully invertible subsets are easy to prove.
\begin{proposition}\label{irrinv P}
Suppose that $S$ and $T$ are two fully invertible  subsets of $A$. Then
\begin{enumerate}
\item    $\text{U}(A)\subset S$.
\item $S$ is multiplicatively closed, i.e. if $s,t\in S$ then
 $st\in S$.
\item If $aba\in S$ for some $a,b\in A$ then $a,b\in S$. In particular,
we have $A_{aba}=A_{\{a,b\}}$.
\item $S\subset T$ iff there is a homomorphism $A_S\to A_T$ making the
following diagram commutative
\begin{displaymath}
\xymatrix{ & A \ar[dr]_{i_T}  \ar[r]^{i_S} & A_S \ar[d] \\
           &   & A_T}
 \end{displaymath}
\end{enumerate}
\end{proposition}
\begin{proof}
1. Trivial.\\
2. It follows from the facts that $S=A\cap \text{U}(A_S)$ and the product
of two invertible   elements is invertible.\\
3. It follows from the simple fact that   $xyx$ ($x,y\in B$) is invertible in a ring $B$ iff
 both $x$ and $y$ are invertible in $B$.\\
4. Easy.
\end{proof}
\begin{remark}
\begin{enumerate}
\item Part 1 of the above proposition tells us that $\text{U}(A)$ is the smallest fully
invertible subset of $A$.
\item Note that part 3 of the above proposition cannot be replaced with
``if $ab\in S$ for some $a,b\in A$ then $a,b\in S$'' (but it is true for
commutative rings).
\end{enumerate}
\end{remark}
\begin{proposition}\label{smallestfully invertible P}
\begin{enumerate}
\item For any family $\{S_i\}$ of fully invertible subsets of $A$,
the subset $\underset{i}{\cap}S_i$ is fully invertible.
\item For any properly invertible subset $S$ of $A$, the subset $A\cap \text{U}(A_S)$ is
the smallest
fully invertible subset of $A$ which contains $S$.
\end{enumerate}
\end{proposition}
\begin{proof}
1. It is easy to see that $a\in A$ is invertible in
$A_{\underset{i}{\cap}S_i}$ iff it is invertible
in every $A_{S_i}$ iff it belongs to every
$S_i$ because every $S_i$ is fully invertible.
 So $\underset{i}{\cap}S_i=A\cap \text{U}(A_{\underset{i}{\cap}S_i})$
is fully invertible.\\
2. Clearly $A\cap \text{U}(A_S)$ is fully invertible and $S\subset A\cap
\text{U}(A_S)$. Moreover
it follows form the universal property of $A_S$ that any fully invertible subset
of $A$ which contains
$S$, must contain $A\cap \text{U}(A_S)$ as well.
So $A\cap \text{U}(A_S)$ is the smallest fully invertible subset containing $A$.
\end{proof}
The following proposition explains the behavior of fully invertible subsets
under ring homomorphisms.
\begin{proposition}\label{pull back P}
Suppose that $\phi:A\to B$ is a ring homomorphism and $S$ is a fully invertible
subset of $B$.
Then $ A\cap S=\phi^{-1}(S)$ is a fully invertible subset of $A$.
Moreover there is a unique homomorphism  $\phi^S:A_{A\cap S}\to B_S$ making the following
diagram commutative
\begin{displaymath}
\xymatrix{ & A \ar[d]_{i_{A\cap S}}  \ar[r]^{\phi} & B \ar[d]^{i_S} \\
           & A_{A\cap S}\ar[r]^{\phi^S}  & B_S}
 \end{displaymath}
\end{proposition}
\begin{proof}
Since $S$ is a fully invertible subset of $B$, there is a ring homomorphism
$\psi:B\to C$ ($C\neq 0$) such that $B\cap \text{U}(C)=S$. Then
$\psi\phi:A\to C$ is a  ring homomorphism and $A\cap \text{U}(C)=A\cap S$.
Hence $A\cap S$ is fully invertible. In particular,
if $\psi=i_S:B\to B_S$ then $i_S\phi:A\to B_S$ sends $A\cap S$ to $\text{U}(B_S)$ and
 hence there is a unique homomorphism $\phi^S:A_{A\cap S}\to B_S$
 making the following diagram commutative
\begin{displaymath}
\xymatrix{ & A \ar[d]_{i_{A\cap S}}  \ar[r]^{\phi} & B \ar[d]^{i_S} \\
           & A_{A\cap S}\ar[r]^{\phi^S}  & B_S}
 \end{displaymath}
Note that $\phi^S=(i_S\phi)_{A\cap S}$.
\end{proof}
\begin{example}
Let  $k\langle x, y \rangle$ be the free $k$-algebra generated by
two elements where $k$ is a commutative field.
Suppose that $\phi:k\langle x, y \rangle \to k$ is the natural
$k$-algebra homomorphism such that $\phi(x)=\phi(y)=0$.
Clearly $k\setminus \{0\}$ is a fully invertible subset of
$k$. So $\phi^{-1}(k\setminus \{0\})$ is a fully invertible subset
of  $k\langle x, y \rangle$. Note that
$ \phi^{-1}(k\setminus \{0\})$
 is the set of all elements in
$k\langle x, y \rangle$
whose constant terms are not zero.
\end{example}
\textbf{Quasi-nilpotent elements:}
suppose that $A$ is a ring and $a\in A$. It might happen that $A_a$ is the zero
ring which
in this case, the element $a\in A$ is called  \textbf{quasi-nilpotent}
\underline{in $A$}.
Clearly, $a\in A$ is   quasi-nilpotent in $A$ if for any ring homomorphism
$\phi:A\to B$ ($B\neq 0$),
 $\phi(a)$   is not invertible in $B$. In the following lemma  another
characterization of quasi-nilpotent elements is given.
\begin{lemma}\label{quasi-fieldmodule L}
An element $a\in A$ is quasi-nilpotent in $A$ iff for any nonzero left
(or right) module $M$ over $A$,
either $aM\neq M$ or $\text{Ann}_M(a)\neq 0$ where $\text{Ann}_M(a)=
\{m\in M|\, am=0\, \}$.
\end{lemma}
\begin{proof}
I prove the equivalent statement `` $a\in A$ is not quasi-nilpotent in $A$ iff there is
a nonzero module $M$
over $A$ such that $a:M\to M$ is an isomorphism.'' If $a$ is not quasi-nilpotent in
$A$ then $M=A_a$
does the job. Conversely if such a module $M$ exists then there is the  homomorphism
$\phi:A\to \text{End}(M)$
coming from the module structure of $M$,
such that $\phi(a)$ is invertible in \text{End}(M). So $a$ is not quasi-nilpotent in $A$. Here
$  \text{End}(M)$ is the ring of
$\mathbb{Z}$-module homomorphisms $M\to M$.
\end{proof}
It is easy to see that the image of a quasi-nilpotent element is quasi-nilpotent.
More precisely,
suppose that $\phi:A\to B$ is a ring homomorphism. If $a\in A$ is quasi-nilpotent in
$A$ then $\phi(a)$ is quasi-nilpotent in $B$. In fact if $\phi(a)$ is not
quasi-nilpotent in $B$ then $B_{\phi(a)}$ is a
nonzero ring in which $a$ is invertible.
\begin{example}
Clearly any nilpotent element is quasi-nilpotent. The converse is true
in the commutative case. But the converse is generally not valid
in the noncommutative case. For example, the matrix
$$X=\begin{pmatrix}
1 & 0\\
0 & 0
\end{pmatrix}$$
is quasi-nilpotent in $\text{M}_2(\mathbb{Q})$ but not nilpotent. Note  that $X$
is not quasi-nilpotent \underline{in the subring of diagonal matrices}. So
``being quasi-nilpotent
or not''
is not universal in the sense that it depends on the ring in which the element is considered.
\end{example}
\textbf{Local homomorphisms:}
 A  ring homomorphism $\phi:A\to B$ is called  \textbf{local} if $
 \text{U}(A)=\phi^{-1}(\text{U}(B))$. In other words,  $\phi:A\to B$
 is local if $\phi(a)\in \text{U}(B)$ implies that $a\in\text{U}(A)$
 for any $a\in A$.
The following proposition, regarding local maps, is easy to prove.
\begin{proposition}\label{localmaps P}
Suppose that $\phi:A\to B$ and $\psi:B\to C$ are ring homomorphisms. Then
\begin{enumerate}
\item If $\phi$ and $\psi$ are local then so is $\psi\phi$.
\item If $\psi\phi$ is local then so is $\phi$.
\end{enumerate}
\end{proposition}
\begin{proof}
\begin{enumerate}
\item  We have
$$(\psi\phi)^{-1}(\text{U}(C))=\phi^{-1}\psi^{-1}(\text{U}(C))=\phi^{-1}(\text{U}(B))
=\text{U}(A).$$
\item  If $\phi(a)\in\text{U}(B)$ then $\psi\phi(a)\in\text{U}(C)$.
Since $\psi\phi$
is local, we must have $a\in\text{U}(A)$.
\end{enumerate}
\end{proof}
\begin{lemma}\label{localideal L}
Suppose that $I$ is an ideal of the ring $A$. Then the quotient
homomorphism $\pi: A\to A/I$
is a local homomorphism iff $I\subset \text{J}(A)$.
Here $\text{J}(A)$ is the Jacobson radical of
$A$.
\end{lemma}
\begin{proof}
First suppose that $I\subset \text{J}(A)$. It is easy to see that  $\phi(a)$ is
invertible in $A/I$
iff there is some $b\in A$ such that $ab,ba\in 1+I$. Since $I\subset \text{J}(A)$,
we have $1+I\subset \text{U}(A)$.
So $ab,ba\in \text{U}(A)$ which implies that $a\in \text{U}(A)$. Conversely, if $\pi$
is local then
$1+I\subset \text{U}(A)$. But it is easy to see that this implies that
$I\subset \text{J}(A)$, see \cite{La} (note
 that, in general, $I\subset \text{J}(A)$ iff $1+I\subset \text{U}(A)$).
\end{proof}
\begin{proposition}\label{kernel P}
Suppose that $\phi:A\to B$ is a local ring homomorphism. Then
$\phi^{-1}(\text{J}(B))\subset \text{J}(A)$. In particular, we have $\ker (\phi)\subset \text{J}(A)$.
\end{proposition}
\begin{proof}
There is the following commutative diagram
\begin{displaymath}
\xymatrix{ & A \ar[d]^{\pi}  \ar[r]^{\phi} & B \ar[d]^{\pi'} \\
           & A/\phi^{-1}(\text{J}(B)) \ar[r]^{\psi}  & B/\text{J}(B)}
 \end{displaymath}
where $\psi$ is the induced homomorphism by $\phi$ and, $\pi$ and $\pi'$ are
 the quotient homomorphisms. By Lemma
\ref{localideal L}, $\pi'$ is local. So $\psi\pi=\pi' \phi$ is local by Proposition
 \ref{localmaps P} part 1. By Proposition \ref{localmaps P}
 part 2,   $\pi$ is local. So by Lemma \ref{localideal L},
  $\phi^{-1}(\text{J}(B))\subset \text{J}(A)$.
\end{proof}
\textbf{Self-localized rings:} A nonzero ring $A$ is called  \textbf{self-localized}
if the only
fully invertible subset of $A$ is $\text{U}(A)$.
In other words, a ring $A$ is self-localized iff any   homomorphism $A\to B$,
from $A$ to a nonzero ring $B$, is a local homomorphism.
Equivalently, $A$ is  self-localized iff  any $a\in A$  is either invertible
or quasi-nilpotent in $A$.

It is easy to see that any homomorphic image of a self-localized ring is self-localized.
This is because the image of an invertible (quasi-nilpotent) element
is invertible (quasi-nilpotent).
\begin{proposition}\label{radical P}
Suppose that $A$ is a self-localized ring. Then $A/\text{J}(A)$ is a simple ring.
\end{proposition}
\begin{proof}
If $A$ is self-localized then for any ideal $I$ of $A$, the quotient
homomorphism $A\to A/I$ is a local homomorphism. So by
Lemma \ref{localideal L}, $I\subset \text{J}(A)$. Therefore, $A/\text{J}(A)$ is a simple ring.
\end{proof}
Clearly any division ring is self-localized. But there are self-localized
rings which are not division rings.
Recall that an element $a\in A$ in the ring $A$ is called \textbf{regular} if it is neither a
right nor a left zero divisor.
The ring $A$ is called
a \textbf{ring of quotients} if every regular element of $A$ is invertible.
It is easy to see that
any simple ring of quotients  is self-localized (because the kernel of $A\to A_a$ is
nonzero if $a$ is not regular).
A subclass of simple quotient rings is the class of simple von Neumann regular rings.
By definition,
  $A$ is called von Neumann regular if for any $a\in A$, there is some $x\in A$ such that
$axa=a$. Moreover we have
\begin{proposition}  \label{vonNeumann P}
A von Neumann regular ring is self-localized iff it is a simple ring.
\end{proposition}
\begin{proof}
Suppose that $A$ is a von Nuemann  regular ring. If $A$ is self-localized
then by Proposition \ref{radical P},
$A/\text{J(A)}$ is a simple ring. But it is well known that $\text{J}(A)=0$ for a von
Neumann regular ring, see \cite{La}.
Conversely, suppose that $A$ is a simple von Nuemann  regular ring. Let
$a\in A\setminus\text{U}(A)$.
Since $axa=a$ for some $x\in A$, the element $a$ is either a left or a right zero divisor.
So $A$ is a simple
 ring of quotients, hence it is self-localized.
\end{proof}
Note that the proof of this lemma shows that any localization ring of a
von Neumann ring is a homomorphic image of it.
\begin{example}
An example of a simple self-localized ring is the ring $\text{M}_n(k)$ of
$n$ by $n$ matrices over a commutative field $k$.

\end{example}

\end{subsection}

%%%%%%%%%%%%%%%%%%%%%%%%%%%%%%%%%%%%%%%%%%%%%%%%%%%%%%%%%%%%%%%%%%%%%%%%%%%%%%%%%%%%%%%%%%%%%%%%
%%%%%%%%%%%%%%%%%%%%%%%%%%%%%%%%%%%%%%%%%%%%%%%%%%%%%%%%%%%%%%%%%%%%%%%%%%%%%%%%%%%%%%%%%%%%%%%%
%%%%%%%%%%%%%%%%%%%%%%%%%%%%%%%%%%%%%%%%%%%%%%%%%%%%%%%%%%%%%%%%%%%%%%%%%%%%%%%%%%%%%%%%%%%%%%%%

\begin{subsection}{Successive localizations}

By Proposition \ref{pull back P}, given a ring homomorphism $\phi:A\to B$,
the subset $S_1=A\cap \text{U}(B)$ is a fully invertible subset of
$A$ and there is  a natural homomorphism $\phi_1=\phi^{S_1}:A_1=A_{S_1}\to B$ which
extends $\phi$.
It might happen that $S_1=\text{U}(A)$. In this case,
  $\phi$ is a local homomorphism. Now suppose that $\phi$ is not local. Then consider
  $\phi_1:A_1\to B$.
  The homomorphism $\phi_1$
might not be local. Set $S_2=A_1\cap \text{U}(B)$ and consider the natural homomorphism
$\phi_2=(\phi_1)^{S_1}:A_2=(A_1)_{S_1}\to B$. Continuing this process, we obtain a sequence
$\mathcal{S}=(S_1,S_2,...)$ of fully invertible subsets of rings $A_{n+1}=(A_n)_{S_{n+1}}$,
 where $A_0=A$.
Note we have two possibilities here. Either $\mathcal{S}$ terminates,  meaning that there is
some $n$ for which the
homomorphism $\phi_n:A_n\to B$ is local, or $\mathcal{S}$ does not terminate meaning that
none
of the homomorphisms
$\phi_n:A_n\to B$ are local.  In any case,  the sequence $\mathcal{S}$ is called a
\textbf{fully invertible system}
on $A$ (coming from $\phi$). Given a fully invertible system $\mathcal{S}=(S_1,S_2,...)$,
define the rings
$A_{\mathcal{S},n}$ inductively via $A_{\mathcal{S},n}=(A_{\mathcal{S},n-1})_{S_n}$
and $A_{\mathcal{S},0}=A$.
Note that any fully invertible subset $S$ of $A$ can be canonically considered as a
fully invertible system on $A$.
More precisely, $S$ can be identified with the fully invertible system
$(S,\text{U}(A_S),\text{U}(A_S),...)$. This identification will be used from now on.

When working with fully invertible systems, the method of proof by induction is
sometimes used which is justified by the
following easy lemma.
\begin{lemma}
Let $\mathcal{S}=(S_1,S_2,...)$ be a fully invertible system on $A$. Then for each $n$, $\mathcal{S}_n=(S_{n+1},S_{n+2},...)$
is a fully invertible system
on $A_{\mathcal{S},n}$.
\end{lemma}
In the following proposition
  a characterization of fully invertible systems is given.
\begin{proposition}
Let $A$ be a ring and $\mathcal{S}=(S_1,S_2,...)$ be a sequence of fully
invertible subsets $S_n$ of rings
$A_{\mathcal{S},n}=(A_{\mathcal{S},n-1})_{S_{n}}$ where $n\geq 1$ and $A_{\mathcal{S},0}=A$.
Then $\mathcal{S} $ is a fully invertible system on $A$
iff    $A_{\mathcal{S},n}\cap S_m=S_{n+1}$
 for each $m> n\geq 0$.
\end{proposition}
\begin{proof}
One direction is easy. Now suppose that $\mathcal{S}$ has the stated property.  Let
$A_{\mathcal{S}}=\underset{n}{\underrightarrow{\lim}} A_{\mathcal{S},n}$
be the direct limit of  the system $i_{S_n}:A_{\mathcal{S},n-1}\to A_{\mathcal{S},n}$.
 Clearly we have a natural homomorphism
 $i_{\mathcal{S}}:A\to A_{\mathcal{S}}$. I claim that $\mathcal{S}$ is
 the fully invertible system coming from $i_\mathcal{S}$.
 Let $ \mathcal{T}=(T_1,T_2,...)$ be the fully invertible system coming from $i_\mathcal{S}$.
Since $T_1=A\cap \text{U}(A_{\mathcal{S}})$, we have $S_1\subset T_1$.
If $a\in T_1$ then $a$ is invertible in $A_{\mathcal{S}}$, hence invertible
in some $A_{\mathcal{S},n}$. This means that
$a\in A\cap S_{n+1}=S_1$ which implies that
$S_1=T_1$. A simple induction shows that $S_n=T_n$, hence $\mathcal{S}=
 {\mathcal{T}}$ is a  fully invertible system on $A$.
\end{proof}
The ring $A_{\mathcal{S}}$ is called the
\textbf{localization} of $A$ at $\mathcal{S}$. Using the presentations of
$A_{S_n}$'s, one can give a presentation of   $A_{\mathcal{S}}$. Roughly speaking,
the ring $A_{\mathcal{S}}$
has the following presentation
$$A_{\mathcal{S}}=\langle A,\{x_{s}\}_{s\in \cup_n S_n}|\, x_ss=sx_s=1,\,
\text{for any}\, s\in \cup_n S_n \rangle. $$
Note that we do not have to adjoin symbols $s\in S_{n}$ (for $n>1$) to this
presentation, because
there are other relations which are hidden in this presentation and realize
$s\in S_n$ ($n>1$). More precisely
each $s\in S_{n}$ ($n>1$) can be  presented as a noncommutative polynomial in terms
of $x_{t_1},...,x_{t_k}$ (and elements of
$A$)  for some
$t_1,...,t_k\in \cup_{i=1}^{n-1} S_i$.  In this case, I say that $s$ depends on $t_1,...,t_k$.
In this presentation, every element of $y\in A_\mathcal{S}$ can be presented as follows
\begin{equation}\label{polynomialpresentation}
y=f(a_1,...,a_m,x_{s_{11}},...,x_{s_{1n_1}},x_{s_{21}},...,x_{s_{2n_2}},...,x_{s_{k1}},...,
x_{s_{kn_k}})
\end{equation}
where
$$f\in \mathbb{Z}\langle Y_1,...,Y_m,Z_{11},...,
Z_{1n_1},Z_{21},...,Z_{2n_2},...,Z_{k1},...,Z_{kn_k} \rangle,$$
$$a_i\in A,\,    s_{ij}\in S_i,$$
and each $s_{ij}$ depends on $s_{pq}$'s with $p< i$.
The presentation \ref{polynomialpresentation} is called a \textbf{polynomial presentation}
of $a$.

One can compare fully invertible systems. More  precisely, suppose that
$\mathcal{S}=(S_1,S_2,...)$ and $\mathcal{T}=(T_1,T_2,...)$
are two fully invertible systems on $A$. If
$S_1\subset T_1$ then we have a natural homomorphism $A_{\mathcal{S},1}\to
 A_{\mathcal{T},1}$. If
$S_2\subset A_{\mathcal{S},1}\cap T_2$ then we have a natural homomorphism
$A_{\mathcal{S},2}\to A_{\mathcal{T},2}$ and
so on and so forth.
In general,
I write $\mathcal{S}\subset \mathcal{T}$ (and say $\mathcal{T}$ contains $\mathcal{S}$) if
 $S_n\subset  A_{\mathcal{S},n}\cap T_n$ for any $n$. In this case, we obtain
 a natural homomorphism $A_\mathcal{S}\to A_{\mathcal{T}}$ making
 the following diagram commutative
\begin{displaymath}
\xymatrix{ & A \ar[dr]_{i_\mathcal{T}}  \ar[r]^{i_\mathcal{S}} & A_{\mathcal{S}} \ar[d]  \\
           &   & A_{\mathcal{T}}}
 \end{displaymath}
Conversely, if there is a ring homomorphism  $A_\mathcal{S}\to A_{\mathcal{T}}$
making
 the above diagram commutative then we must have $\mathcal{S}\subset \mathcal{T}$.

Like the localizations at fully invertible subsets, $A_{\mathcal{S}}$ enjoys a universal
property.
More precisely,  $A_{\mathcal{S}}$ has the following universal property:
for any ring homomorphism $\phi:A\to B$ such that $\mathcal{S}$ is contained
in the fully invertible system coming from $\phi$, there is a unique homomorphism
$\phi_\mathcal{S}:A_\mathcal{S}\to B$ such that $\phi=\phi_\mathcal{S}i_\mathcal{S}$, i.e.
the following diagram is commutative
\begin{displaymath}
\xymatrix{ & A \ar[dr]_{\phi}  \ar[r]^{i_\mathcal{S}} & A_{\mathcal{S}}
\ar[d]^{\phi_\mathcal{S}} \\
           &   & B}
 \end{displaymath}
Note that  $A_{\mathcal{S}}$ is completely determined by this universal property.

We can also consider the intersection of (any family of) fully invertible systems.
More precisely, given fully invertible systems $\mathcal{S}$ and $\mathcal{T}$,  there is a
(unique) fully invertible system $\mathcal{S}\cap \mathcal{T}$,
called the intersection of  $\mathcal{S}$ and $\mathcal{T}$, such that for any
fully invertible system
$\mathcal{R}$ on $A$, we have $\mathcal{R}\subset \mathcal{S}$ and $\mathcal{R}
\subset \mathcal{T}$ iff
$\mathcal{R}\subset \mathcal{S}\cap \mathcal{T}$.

Now we have the following proposition similar to Proposition \ref{pull back P}.
\begin{proposition}\label{pullbacksystem P}
Suppose that $\phi:A\to B$ is a ring homomorphism and $\mathcal{S}=(S_1,S_2,...)$
is a fully invertible system on $B$.
Then  the sequence
$$\phi^{*}(\mathcal{S}):=(A\cap S_1,A_{A\cap S_1}\cap S_2,...)$$
is a fully invertible system on $A$, called
the \textbf{pull back} of $\mathcal{S}$ to $A$. Moreover
\begin{enumerate}
\item There  is  a unique ring homomorphism  $\phi^\mathcal{S}:A_{\phi^{*}(\mathcal{S})}
\to B_\mathcal{S}$
 making the following diagram commutative
\begin{displaymath}
\xymatrix{ & A \ar[d]_{i_{\phi^{*}(\mathcal{S})}}  \ar[r]^{\phi} & B \ar[d]^{i_\mathcal{S}} \\
           & A_{\phi^{*}(\mathcal{S})}\ar[r]^{\phi^\mathcal{S}}  & B_\mathcal{S}}
 \end{displaymath}
\item The homomorphism  $\phi^\mathcal{S}:A_{\phi^{*}(\mathcal{S})}\to B_\mathcal{S}$
is a local homomorphism.
\end{enumerate}
\end{proposition}
\begin{proof}
Since $\mathcal{S}$ is a fully invertible system on $B$, there is a ring homomorphism
$\psi:B\to C$ ($C\neq 0$) which induces $\mathcal{S}$. Then
 it is easy to see that $\phi^{*}(\mathcal{S})$ is the fully invertible system coming
  from $\psi\phi$.
Hence $\phi^{*}(\mathcal{S})$ is a fully invertible system on $A$. In particular,
if we take $\psi=i_\mathcal{S}:B\to B_\mathcal{S}$ then  $\phi^{*}(\mathcal{S})$ is contained
in the fully invertible system
coming from $i_\mathcal{S}\phi$
  and
 hence we have a unique ring homomorphism $\phi^\mathcal{S}:A_{\phi^{*}(\mathcal{S})}
 \to B_\mathcal{S}$
  making the following diagram  commutative
\begin{displaymath}
\xymatrix{ & A \ar[d]_{i_{\phi^{*}(\mathcal{S})}}  \ar[r]^{\phi} & B \ar[d]^{i_\mathcal{S}} \\
           & A_{\phi^{*}(\mathcal{S})}\ar[r]^{\phi^\mathcal{S}}  & B_\mathcal{S}}
 \end{displaymath}
Finally I show that $\phi^{\mathcal{S}}$ is local. Suppose that
$x\in A_{\phi^{*}(\mathcal{S})}$
is invertible in $B_\mathcal{S}$. Hence $x$ belongs to some $A_{\phi^{*}(\mathcal{S}),n}$
and is invertible in $B_{\mathcal{S},n}$
for $n$ big enough. Then $x\in A_{\phi^{*}(\mathcal{S}),n}\cap S_{n+1}$ which implies that
  $x\in\text{U}(A_{\phi^{*}(\mathcal{S})})$.
\end{proof}
Note that the pull back operation is transitive, i.e.
 if $\phi:A\to B$ and $\psi:B\to C$ are  ring homomorphisms and $\mathcal{S}$ is a fully
  invertible system on $C$ then
  $(\psi\phi)^{*}(\mathcal{S})=\phi^{*}(\psi^{*}(\mathcal{S}))$.
\begin{proposition}\label{F.I.SinA_S P}
For any  fully invertible system $\mathcal{S}$ on $A$, there is a one-to-one correspondence
between fully invertible systems on $A_\mathcal{S}$ and fully invertible systems on $A$
containing $\mathcal{S}$
(given by the pull back
map
$\mathcal{T}\mapsto i_\mathcal{S}^{*}(\mathcal{T})$).
Moreover, for each fully invertible
system $\mathcal{T}$ on
 $A_\mathcal{S}$, the natural ring homomorphism $A_{i_\mathcal{S}^{*}(\mathcal{T})}\to
 (A_\mathcal{S})_{\mathcal{T}}$ is
 an isomorphism.
\end{proposition}
\begin{proof}
Consider the natural homomorphism $i=i_\mathcal{S}:A\to A_{\mathcal{S}}$.
Clearly the pull back of any fully invertible
system on $A_{\mathcal{S}}$ to $A$ contains $\mathcal{S}$. On the other hand
if a fully invertible system $\mathcal{S}'$ on $A$ contains $\mathcal{S}$
then we obtain a natural homomorphism $A_{\mathcal{S}}\to A_{\mathcal{S}'}$.
Then the pull back of the fully invertible system
on $A_{\mathcal{S}}$ coming from  $A_{\mathcal{S}}\to A_{\mathcal{S}'}$, to $A$
is just $\mathcal{S}'$.
So, to show that  $\mathcal{T}\mapsto i^{*}(\mathcal{T})$
gives a one-to-one correspondence,  it is enough to show that this pull back map
is one-to-one. To show this, I first prove that
the natural  homomorphism $A_{i^{*}(\mathcal{T})}\to
 (A_\mathcal{S})_{\mathcal{T}}$ is
 an isomorphism for any fully invertible system $\mathcal{T}$ on
 $A_\mathcal{S}$. Since $\mathcal{S}\subset i^{*}(\mathcal{T})$, there is a natural
 homomorphism
 $A_\mathcal{S}\to A_{i^{*}(\mathcal{T})}$. It is easy to see that this homomorphism
 has the universal property of
 $(A_\mathcal{S})_{\mathcal{T}}$ which,
  in particular, implies that the natural
 homomorphism $A_{i^{*}(\mathcal{T})}\to
 (A_\mathcal{S})_{\mathcal{T}}$ is an isomorphism.
Now, if $\mathcal{T}_1$ and $\mathcal{T}_2$ are two fully invertible systems on
$A_\mathcal{S}$ such that $i^{*}(\mathcal{T}_1)=i^{*}(\mathcal{T}_2)$
then we have an isomorphism $(A_{\mathcal{S}})_{\mathcal{T}_1}\to
(A_{\mathcal{S}})_{\mathcal{T}_2}$ making the following diagram commutative
\begin{displaymath}
\xymatrix{ & A_\mathcal{S} \ar[dr] \ar[r] &  (A_{\mathcal{S}})_{\mathcal{T}_1} \ar[d] &  \\
        &  &  (A_{\mathcal{S}})_{\mathcal{T}_2} &
        }
         \end{displaymath}
because $ A_{i^{*}(\mathcal{T}_i)}\to (A_{\mathcal{S}})_{\mathcal{T}_i}$
are natural isomorphisms. So
$\mathcal{T}_1=\mathcal{T}_2$.
\end{proof}
\begin{example}
There are fully invertible systems which are not fully invertible subsets.
Consider the free division ring $F\langle x, y\rangle$ generated by two elements, see
\cite{Getal} for the definition and the construction of $F\langle x, y\rangle$.
Then the fully invertible system $\mathcal{S}$
on $\mathbb{Q}\langle x, y\rangle$
coming from the natural homomorphism $\mathbb{Q}\langle x, y\rangle\to F\langle x, y\rangle$,
is not a fully invertible subset.
In fact,  $\mathcal{S}$ does not terminate (see the notion of inversion height,
\cite{Getal,Re}).  

\end{example}

\end{subsection}

%%%%%%%%%%%%%%%%%%%%%%%%%%%%%%%%%%%%%%%%%%%%%%%%%%%%%%%%%%%%%%%%%%%%%%%%%%%%%%%%%%%%%%%%%%%%%%%%
%%%%%%%%%%%%%%%%%%%%%%%%%%%%%%%%%%%%%%%%%%%%%%%%%%%%%%%%%%%%%%%%%%%%%%%%%%%%%%%%%%%%%%%%%%%%%%%%
%%%%%%%%%%%%%%%%%%%%%%%%%%%%%%%%%%%%%%%%%%%%%%%%%%%%%%%%%%%%%%%%%%%%%%%%%%%%%%%%%%%%%%%%%%%%%%%%
%%%%%%%%%%%%%%%%%%%%%%%%%%%%%%%%%%%%%%%%%%%%%%%%%%%%%%%%%%%%%%%%%%%%%%%%%%%%%%%%%%%%%%%%%%%%%%%%

\end{section}

%%%%%%%%%%%%%%%%%%%%%%%%%%%%%%%%%%%%%%%%%%%%%%%%%%%%%%%%%%%%%%%%%%%%%%%%%%%%%%%%%%%%%%%%%%%%%%%%
%%%%%%%%%%%%%%%%%%%%%%%%%%%%%%%%%%%%%%%%%%%%%%%%%%%%%%%%%%%%%%%%%%%%%%%%%%%%%%%%%%%%%%%%%%%%%%%%
%%%%%%%%%%%%%%%%%%%%%%%%%%%%%%%%%%%%%%%%%%%%%%%%%%%%%%%%%%%%%%%%%%%%%%%%%%%%%%%%%%%%%%%%%%%%%%%%
%%%%%%%%%%%%%%%%%%%%%%%%%%%%%%%%%%%%%%%%%%%%%%%%%%%%%%%%%%%%%%%%%%%%%%%%%%%%%%%%%%%%%%%%%%%%%%%%
%%%%%%%%%%%%%%%%%%%%%%%%%%%%%%%%%%%%%%%%%%%%%%%%%%%%%%%%%%%%%%%%%%%%%%%%%%%%%%%%%%%%%%%%%%%%%%%%
%%%%%%%%%%%%%%%%%%%%%%%%%%%%%%%%%%%%%%%%%%%%%%%%%%%%%%%%%%%%%%%%%%%%%%%%%%%%%%%%%%%%%%%%%%%%%%%%
%%%%%%%%%%%%%%%%%%%%%%%%%%%%%%%%%%%%%%%%%%%%%%%%%%%%%%%%%%%%%%%%%%%%%%%%%%%%%%%%%%%%%%%%%%%%%%%%
%%%%%%%%%%%%%%%%%%%%%%%%%%%%%%%%%%%%%%%%%%%%%%%%%%%%%%%%%%%%%%%%%%%%%%%%%%%%%%%%%%%%%%%%%%%%%%%%

\begin{section}{Definition of F-schemes}
In this part F-schemes are defined. It is assumed that the reader is familiar
with basics of
the theory of ringed spaces and  schemes, see \cite{Ha} for example.

%%%%%%%%%%%%%%%%%%%%%%%%%%%%%%%%%%%%%%%%%%%%%%%%%%%%%%%%%%%%%%%%%%%%%%%%%%%%%%%%%%%%%%%%%%%%%%%%
%%%%%%%%%%%%%%%%%%%%%%%%%%%%%%%%%%%%%%%%%%%%%%%%%%%%%%%%%%%%%%%%%%%%%%%%%%%%%%%%%%%%%%%%%%%%%%%%
%%%%%%%%%%%%%%%%%%%%%%%%%%%%%%%%%%%%%%%%%%%%%%%%%%%%%%%%%%%%%%%%%%%%%%%%%%%%%%%%%%%%%%%%%%%%%%%%

\begin{subsection}{The full spectrum of a ring}
Suppose that $A$ is a (nonzero) ring. Let $\text{F}(A)$
be the set of all   fully invertible systems on $A$.   The set
$\text{F}(A)$ is called the \textbf{full spectrum} of $A$.
Clearly $ \text{F}(A)\neq \emptyset$. First, I want to introduce a topology on $\text{F}(A)$. \\
By a fraction of $A$, I mean an $n$-tuple  $\begin{bf} a \end{bf} =(a_1,...,a_n)$ such that
$$a_1\in A, a_2\in A_{a_1},..., a_n\in (...(A_{a_1})_{a_2}...)_{a_{n-1}}.$$
For a fraction $\begin{bf} a \end{bf}=(a_1,...,a_n)$ of $A$, set $A_{\begin{bf}
a \end{bf},0}=A$ and $A_{\begin{bf} a
\end{bf},k}=(A_{\begin{bf} a \end{bf},k-1})_{a_k}$ inductively for $k=1,...,n$.
Denote $A_{\begin{bf} a \end{bf},n}$ by $A_{\begin{bf} a \end{bf}}$.
Given a fraction $\begin{bf} a \end{bf}=(a_1,...,a_n)$ and a fully invertible system
$\mathcal{S}=(S_1,S_2,...)$, I write
$\begin{bf} a \end{bf}\in \mathcal{S}$
if for any $1\leq i\leq n$, we have $a_i\in S_i$.  In other words, $\begin{bf} a
\end{bf}\in \mathcal{S}$ iff $a_i\in \text{U}(A_\mathcal{S})$
for any $1\leq i\leq n$. It is easy to see that $\begin{bf} a \end{bf}\in
\mathcal{S}$ iff there is a homomorphism
$A_{\begin{bf} a \end{bf}}\to A_\mathcal{S}$ making the following diagram commutative
\begin{displaymath}
\xymatrix{ & A \ar[dr]  \ar[r] & A_{\begin{bf} a \end{bf}} \ar[d]  \\
           &   & A_{\mathcal{S}}}
 \end{displaymath}
For  a fraction $\begin{bf} a \end{bf}=(a_1,...,a_n)$  of $A$,   set
$$D(\begin{bf} a \end{bf})=D(a_1,...,a_n)=\{\mathcal{S}\in\text{F}(A)|\, \begin{bf} a
\end{bf}\in \mathcal{S} \}.$$
Clearly $D(a_1,...,a_n)=D(a_1,...,a_n,1,...,1)$.
\begin{proposition}\label{fundamentalopensubsets P}
\begin{enumerate}
\item $D(1)=\text{F}(A)$ and $D(0)=\emptyset$.
\item  Let   $\begin{bf} a \end{bf}=(a_1,...,a_n)$ and $\begin{bf} b \end{bf}
 =(b_1,...,b_n)$ be two fractions of $A$. Then
 $\langle \begin{bf} a \end{bf},\begin{bf} b \end{bf}\rangle :=
 (a_1b_1a_1,...,a_nb_na_n)$ is a (well-defined) fraction of $A$ and  $D(\begin{bf}
  a \end{bf})\cap D(\begin{bf} b \end{bf})=D(\langle \begin{bf} a \end{bf},
  \begin{bf} b \end{bf}\rangle)$.
\item  Let   $\begin{bf} a \end{bf}=(a_1,...,a_n)$ and $\begin{bf} b \end{bf}
=(b_1,...,b_n)$ be two fractions of $A$. We have
$D(\begin{bf} b \end{bf})\subset D(\begin{bf} a \end{bf})$ iff  there is a homomorphism
$A_{\begin{bf} a \end{bf}}\to A_{\begin{bf} b \end{bf}}$ making the following
diagram commutative
\begin{displaymath}
\xymatrix{ & A \ar[rd]  \ar[r] & A_{\begin{bf} a \end{bf}} \ar[d]  \\
           &   & A_{\begin{bf} b \end{bf}}}
 \end{displaymath}
\end{enumerate}
\end{proposition}
\begin{proof}
The first  part  is trivial and the second part follows form
Proposition \ref{irrinv P} part 3. One direction of the last part is trivial.
To prove the other direction, suppose that $D(\begin{bf} b \end{bf})
\subset D(\begin{bf} a \end{bf})$.
If there is not any natural
homomorphism $A_{\begin{bf} a \end{bf}}\to A_{\begin{bf} b \end{bf}}$,
then  there is some $a_i$ which is not invertible
in $A_{\begin{bf} b \end{bf}}$. But then  there is
some $\mathcal{S}\in D(\begin{bf} b \end{bf})$ such that $a_i$ is not
invertible in $A_\mathcal{S}$. This means that
$\mathcal{S}\notin D(\begin{bf} a \end{bf})$, a contradiction.
\end{proof}
By part 2 of the above proposition, the subsets $D(\begin{bf} a \end{bf})$'s
form a base for a topology
  on $\text{F}(A)$.   From now on,   consider $\text{F}(A)$ as
  a topological space with this topology.   Open subsets of the from
  $D(\begin{bf} a \end{bf})$,
 are called \textbf{fundamental} open subsets.
 %For $\mathcal{T}\in \text{F}(A)$, set
%$V(\mathcal{T})=\{\mathcal{S}\in \text{F}(A)| \mathcal{S}\subset \mathcal{T}\}.$
\begin{proposition}\label{closure of a point P}
For any $\mathcal{T}\in \text{F}(A)$, we have $\overline{\{\mathcal{T}\}}=\{\mathcal{S}\in \text{F}(A)| \mathcal{S}\subset \mathcal{T}\}$.
\end{proposition}
\begin{proof}
We have $\mathcal{S}\in \overline{\{\mathcal{T}\}}$ iff  $\mathcal{S}\in D(\begin{bf} a \end{bf})$
implies that $\mathcal{T}\in D(\begin{bf} a \end{bf})$ for any fraction of
$A$. This implies that
$S_n\subset T_n\cap A_{\mathcal{S},n}$ for any $n$, hence we have
$\mathcal{S}\in \overline{\{\mathcal{T}\}}$
iff $\mathcal{S}\subset \mathcal{T}$.
\end{proof}
One consequence of this proposition is that
if an open subset of $\text{F}(A)$
contains $\mathcal{S}$ then it must contain any $\mathcal{S}\subset \mathcal{T}$.
This proposition also implies that every $\text{F}(A)$ is a $\text{T}_0$ topological space,
i.e. for any points $x,y\in \text{F}(A)$, there is an open subset containing one of the them,
but not the other one.

The topological space $\text{F}(A)$ has  $\text{U}(A)$ as one of its points.
This point is the minimal point of $\text{F}(A)$, i.e.
any point of $\text{F}(A)$ contains this point. In other words, the only
open subset of $\text{F}(A)$ which contains $ \text{U}(A)$ is
the whole space $\text{F}(A)$. Equivalently, every nonempty closed subset of $\text{F}(A)$
contains $ \text{U}(A) $, and moreover,  the point $ \text{U}(A)$ is the only
 closed point of $\text{F}(A)$, because
it is contained in any fully invertible system. In \cite{Gr}, such a point is
called the center of $\text{F}(A)$.
More precisely
\begin{definition}
Suppose that $X$ is a topological space. A point $x\in X$ is called the \textbf{center} of $X$
if $x$ is a closed point and contained in any nonempty closed subspace of $X$.
\end{definition}
Clearly, if the center exists then it is unique.
For any ring $A$, the point $\text{U}(A)$ is the center of  $\text{F}(A)$.
Similarly, every fundamental open subset $D(\begin{bf} a \end{bf})$ (with its
induced topological structure)
has a center. More precisely, $A\cap \text{U}(A_{\begin{bf} a \end{bf}})$ is
the center of $D(\begin{bf} a \end{bf})$.
Moreover
\begin{proposition}\label{subsetswithcenter P}
An open subset of   $\text{F}(A)$ has a center (in its induced topology) iff it
is a fundamental open subset of
$\text{F}(A)$.
\end{proposition}
\begin{proof}
If $x$ is the center of an open subset $U$ of   $\text{F}(A)$ then  there is
some fundamental open subset
  $D(\begin{bf} a \end{bf})\subset U$ which contains $x$, because fundamental
  open subsets form a base for the topology
  of  $\text{F}(A)$. Since $x$ is the center of $U$, we must have
  $U\subset D(\begin{bf} a \end{bf})$ as well which implies that
  $U=D(\begin{bf} a \end{bf})$.
\end{proof}
It is obvious that any topological space with a center is quasi-compact.
So for any ring $A$, the topological space
$\text{F}(A)$ is quasi-compact. Moreover, every open covering of $\text{F}(A)$ has to have
$\text{F}(A)$ as one of its elements.
\begin{remark}
For any commutative local ring $A$, the affine
scheme $\text{Spec}(A)$ has a center. It is easy to see that if a scheme has a
center then it has to be
isomorphic to $\text{Spec}(A)$ for some local ring $A$. Topologically, the
full spectrum of rings
resembles the prime spectrum of   commutative local rings. The fact that
any affine F-scheme has a center, simplifies a lot of arguments regarding F-schemes, whereas
the corresponding arguments for schemes might be much harder to present.
\end{remark}

\end{subsection}

%%%%%%%%%%%%%%%%%%%%%%%%%%%%%%%%%%%%%%%%%%%%%%%%%%%%%%%%%%%%%%%%%%%%%%%%%%%%%%%%%%%%%%%%%%%%%%%%
%%%%%%%%%%%%%%%%%%%%%%%%%%%%%%%%%%%%%%%%%%%%%%%%%%%%%%%%%%%%%%%%%%%%%%%%%%%%%%%%%%%%%%%%%%%%%%%%
%%%%%%%%%%%%%%%%%%%%%%%%%%%%%%%%%%%%%%%%%%%%%%%%%%%%%%%%%%%%%%%%%%%%%%%%%%%%%%%%%%%%%%%%%%%%%%%%
%%%%%%%%%%%%%%%%%%%%%%%%%%%%%%%%%%%%%%%%%%%%%%%%%%%%%%%%%%%%%%%%%%%%%%%%%%%%%%%%%%%%%%%%%%%%%%%%

\begin{subsection}{F-schemes and their structure sheaf  }
Suppose that $A$ is a ring. I want to define a sheaf $\mathcal{O}=\mathcal{O}_A $
of (possibly noncommutative)   rings on $\text{F}(A)$.
For any open subset $U\subset\text{F}(A)$,   define $\mathcal{O}(U)$
to be the set of all functions $t:U\to\underset{\mathcal{S}\in U}{\cup}A_{\mathcal{S}}$
such that\\
\begin{enumerate}
\item for any $\mathcal{S}\in U$, we have $t(\mathcal{S})\in A_{\mathcal{S}}$, and
\item for any $\mathcal{S}_0\in U$, there are an open neighborhood $U_0\subset U$
of $\mathcal{S}_0$,
an element
$$f\in \mathbb{Z}\langle Y_1,...,Y_m,Z_{11},...,Z_{1n_1},Z_{21},...,Z_{2n_2},...,
Z_{k1},...,Z_{kn_k} \rangle$$
and elements $a_i\in A$ and $s_{ij}\in S_i$,
(for any $\mathcal{S}=(S_1,S_2,...)\in U_0$)
such that,
as an element of $A_\mathcal{S}$, for any $\mathcal{S}=(S_1,S_2,...)\in U_0$,
\begin{equation}\label{localpresentation}
t(\mathcal{S})=f(a_1,...,a_m,x_{s_{11}},...,x_{s_{1n_1}},x_{s_{21}},...,
x_{s_{2n_2}},...,x_{s_{k1}},...,x_{s_{kn_k}})
\end{equation}
and $s_{ij}$ depends on $s_{pq},p<i$.
\end{enumerate}
It is clear that the sum and product of two such functions are a function of this form.
Moreover for any open subsets $U\subset V$, we have the  restriction map
 $\mathcal{O}(V)\to\mathcal{O}(U)$. Therefore we have
a presheaf of  (possibly noncommutative)   rings on $\text{F}(A)$, which  is in fact a sheaf.
So $(\text{F}(A),\mathcal{O})$ is a  ringed space.
Recall that a ringed space is a pair $(X,\mathcal{O}_X)$ of a topological space $X$
and a sheaf of
(possibly noncommutative)   rings. The sheaf $\mathcal{O}_X$ is called
the structure sheaf of $X$. A \textbf{ringed subspace} of a ringed space
 $(X,\mathcal{O}_X)$ is a ringed space $(Y,\mathcal{O}_{X}|_{Y})$
where $Y$ is a subspace of $X$ and $\mathcal{O}_{X}|_Y$ is the restriction
of the structure sheaf of $X$ to $Y$.
Sometimes,   a ringed space $(X,\mathcal{O}_X)$ is simply shown  by $X$.

We note that from the definition of
the structure sheaf of $\text{F}(A)$, any section of $\mathcal{O}_A$,
is locally of the form $b\in A_{\begin{bf} a \end{bf}}$ for some fraction
$\begin{bf} a \end{bf}$ of $A$.
More precisely, if $t$ is a section of $\mathcal{O}_A$ over an open subset $U$,
then for any $\mathcal{S}_0$, we can find a fraction $\begin{bf} a \end{bf}$
of $A$ (such that $\mathcal{S}_0\in D(\begin{bf} a \end{bf})\subset U$) and an
element $b\in A_{\begin{bf} a \end{bf}}$ such that for any $\mathcal{S}\in
D(\begin{bf} a \end{bf})$ we have $t(\mathcal{S})=b$
as elements of $\mathcal{S}$. This is because in the polynomial presentation of
$t$ there are only finitely many $x_{ij}$'s.
In the following proposition,  some properties of the structure sheaf of the ringed
 space $\text{F}(A)$ are given.
\begin{proposition}\label{structuresheaf P}
\begin{enumerate}
\item At the level of stalks, for any $\mathcal{S}\in \text{F}(A)$, we have an
isomorphism $\mathcal{O}_{A,\mathcal{S}}\cong A_{\mathcal{S}}$.
\item For any fraction $\begin{bf} a \end{bf}$ of $A$ such that $D(\begin{bf}
a \end{bf})\neq \emptyset$, we have an isomorphism
$\mathcal{O}_A(D(\begin{bf} a \end{bf}))\cong A_{\begin{bf} a \end{bf}}$. In particular, we have
$\Gamma(\text{F}(A),\mathcal{O}_{A})\cong A$.
\end{enumerate}
\end{proposition}
\begin{proof}
1. We have a ring homomorphism $\phi: \mathcal{O}_{A,\mathcal{S}}\to A_\mathcal{S}$
defined by sending
$t\in  \mathcal{O}_{A,\mathcal{S}}$ to $t(\mathcal{S})$.
Any element   $a\in A_\mathcal{S}$ can be presented as in \ref{polynomialpresentation}.
But it is easy to see that
this presentation gives a well-defined element of $A_{\mathcal{T}}$ in a
 neighborhood of $\mathcal{S}$ whose image
is just $a$. So  $\phi$ is onto.
If $t(\mathcal{S})=0$ then since $t$ is, locally at $\mathcal{S}$,
presented as in \ref{localpresentation},
we see that $f=0$ in $A_\mathcal{S}$. This means that
$f=0$ in some $A_{\mathcal{S},n}$. Therefore $f$ is in the ideal of
$A\langle \{x_{s_{ij}}\}_{ij}\rangle$
generated by $s_{ij}x_{s_{ij}}-1$ and $x_{s_{ij}}s_{ij}-1$.
Since we only need finitely many
$s_{ij}x_{s_{ij}}-1$ and $x_{s_{ij}}s_{ij}-1$ to generate $f$,
this is valid in a neighborhood
of $\mathcal{S}$, i.e. $t=0$ in $\mathcal{O}_{A,\mathcal{S}}$. So $\phi$ is  also
injective.\\

\noindent 2. Using part 1, we have
$$ \mathcal{O}_{A}(D(\begin{bf} a \end{bf}))\cong
\mathcal{O}_{A,A\cap \text{U}(A_{\begin{bf} a \end{bf}})}\cong A_{\begin{bf} a \end{bf}}$$
because the only open subset of $D(\begin{bf} a \end{bf})$ containing
$A\cap \text{U}(A_{\begin{bf} a \end{bf}})$ is
$D({\begin{bf} a \end{bf}})$ (recall that $A\cap \text{U}(A_{\begin{bf} a \end{bf}})$
is the center of $D(\begin{bf} a \end{bf})$).

\end{proof}
\begin{remark}\label{definitionofstructuresheaf R}
There is an  equivalent way to define the structure sheaf $\mathcal{O}_A$.
In general if a presheaf (of abelian groups)
is defined on a base of open subsets of a topological space then there is
a unique way to extend this presheaf to a sheaf on the topological space, provided that
it satisfies a ``sheaflike'' condition, see \cite{Gr}. I explain this construction for
$\mathcal{O}_A$.
First define
$\mathcal{O}_A(D(\begin{bf} a \end{bf}))=A_{\begin{bf} a \end{bf}}$. If
$D(\begin{bf} b \end{bf})\subset D(\begin{bf} a \end{bf})$ then consider
the natural homomorphism $ A_{\begin{bf} a \end{bf}}\to A_{\begin{bf} b \end{bf}} $
as in Proposition \ref{fundamentalopensubsets P}.
For any open subset $U$ of $\text{F}(A)$, set
$$\mathcal{O}_A(U)=\underset{D(\begin{bf} a \end{bf})\subset U}{\underleftarrow{\lim}}
 \mathcal{O}_A(D(\begin{bf} a \end{bf}))
 =\underset{D(\begin{bf} a \end{bf})\subset U}{\underleftarrow{\lim}}
 A_{\begin{bf} a \end{bf}}.
 $$
Finally, if $U\subset V$ are two open subsets of  $\text{F}(A)$ then we have a natural
homomorphism
$\mathcal{O}_A(V)\to\mathcal{O}_A(U)$. Since fundamental open subsets have a center,
the ``sheaflike''
condition is always satisfied. So this gives a sheaf of rings on $\text{F}(A)$.
It is easy to see that this definition is equivalent to the above definition  of
the structure sheaf $\mathcal{O}_A$.
\end{remark}
Recall that a morphism between two ringed spaces $(X,\mathcal{O}_X)$ and $(Y,\mathcal{O}_Y)$
is a pair $(f,f^{\sharp})$ of a continuous map $f:X\to Y$ and a homomorphism of sheaves
of rings $f^{\sharp}:\mathcal{O}_Y\to f_*\mathcal{O}_X$ on $Y$. Two ringed spaces
$X$ and $Y$ are isomorphic if there is a morphism $(f,f^{\sharp}):X\to Y$ which has an inverse.
For brevity,  the morphism $(f,f^{\sharp})$ is simply denoted by $f$.
\begin{definition}
A morphism
$$(f,f^{\sharp}):(X,\mathcal{O}_X)\to (Y,\mathcal{O}_Y)$$
between two ringed spaces
is called \textbf{local} if for any
$x\in X$, the natural homomorphism  $f^{\sharp}_x:\mathcal{O}_{Y,f(x)}\to\mathcal{O}_{X,x}$
is local.
\end{definition}
\begin{theorem}\label{AffineF-schemes T}
\begin{enumerate}
\item Any ring homomorphism $\phi: A\to B$ induces a natural  local morphism
$$(f,f^{\sharp}):(\text{F}(B),\mathcal{O}_{B})
\to(\text{F}(A),\mathcal{O}_{A})$$
of ringed spaces.
\item Any   local morphism $(\text{F}(B),\mathcal{O}_{B})
\to(\text{F}(A),\mathcal{O}_{A})$ is induced from a
unique ring homomorphism $A\to B$ as in 1.
\end{enumerate}
\end{theorem}
\begin{proof}
1. By Proposition   \ref{pullbacksystem P}, the assignment
$\mathcal{S}\mapsto \phi^{*}(\mathcal{S})$ does
define a map
$f:\text{F}(B)\to\text{F}(A)$. This map is continuous,
because $f^{-1}(D(\begin{bf} a \end{bf} ))=D(\phi(\begin{bf} a \end{bf}))$
for any fraction $\begin{bf} a \end{bf}=(a_1,...,a_n)$ of $A$
(here $\phi(\begin{bf} a \end{bf})$ is the fraction
$(\phi(a_1),\phi_1(a_2),...,\phi_{n-1}(a_n))$, where $\phi_1:A_{a_1}\to B_{\phi(a_1)}$
is the natural homomorphism
and so on and so forth). Using the way the structure sheaves
$\mathcal{O}_A$  and $\mathcal{O}_B$ are defined and by composing with
the natural homomorphisms $\phi^{\mathcal{S}}$, we can define a morphism
$f^{\sharp}:\mathcal{O}_A\to f_{*}\mathcal{O}_B$ for which
$f^{\sharp}_{\mathcal{S}}=\phi^{\mathcal{S}}$.
So we obtain a morphism
$$(f,f^{\sharp}):(\text{F}(B),\mathcal{O}_{B})
\to(\text{F}(A),\mathcal{O}_{A})$$
of ringed spaces. Since $f^{\sharp}_{\mathcal{S}}=\phi^{\mathcal{S}}$, by Proposition
\ref{pullbacksystem P} part 2, $f$ is a local morphism.  \\
2. By  Proposition \ref{structuresheaf P}, for any morphism
$$(f,f^{\sharp}):(\text{F}(B),\mathcal{O}_{B})\to(\text{F}(A),\mathcal{O}_{A}),$$
we obtain a ring homomorphism
$$\phi:A\cong \Gamma(\text{F}(A),\mathcal{O}_{A})\to
B\cong\Gamma(\text{F}(B),\mathcal{O}_{B}).$$
Then for any $\mathcal{S} \in \text{F}(A)$, we have the following commutative diagram
\begin{displaymath}
\xymatrix{ & A \ar[d]  \ar[r]^{\phi} & B \ar[d] \\
           & A_{f(\mathcal{S})}\ar[r]^{f^{\sharp}_\mathcal{S}}  & B_\mathcal{S}}
 \end{displaymath}
Now the commutativity of the above diagram and the fact that
${f^{\sharp}_\mathcal{S}}$ is a local homomorphism, imply
$f(\mathcal{S})=\phi^{*}(\mathcal{S})$. Finally the above diagram shows
that $(f,f^{\sharp})$ is induced from $\phi$ as in part 1.
\end{proof}
\begin{proposition}      \label{fundamentalopensubfschemes P}
For any fraction $\begin{bf} a \end{bf}$ of $A$ such that  $D(\begin{bf}
a \end{bf})\neq\emptyset$, we have an isomorphism
$(D(\begin{bf} a \end{bf}),\mathcal{O}_{A}|_{D(\begin{bf} a \end{bf})})\cong
 \text{F}(A_{\begin{bf} a \end{bf}})$.
\end{proposition}
\begin{proof}
The natural ring homomorphism $A\to A_{\begin{bf} a \end{bf}}$ gives rise to a local morphism
$f:\text{F}(A_{\begin{bf} a
\end{bf}})\to\text{F}(A)$ of ringed spaces
 by Theorem
\ref{AffineF-schemes T}.  First, we have $f$ is a bijection from $\text{F}(A_{\begin{bf} a
\end{bf}})$ onto  $D(\begin{bf} a \end{bf})$, by Proposition \ref{F.I.SinA_S P}.
It is easy to see that $f$ maps open subsets to open subsets. Hence $f$ is a homeomorphism
from   $\text{F}(A_{\begin{bf} a
\end{bf}})$ onto $D(\begin{bf} a \end{bf})$. Finally, using Proposition \ref{F.I.SinA_S P},
we see that the natural homomorphism
$$f^{\sharp}_{\mathcal{S}}:\mathcal{O}_{A,f(\mathcal{S})}\cong A_{f(\mathcal{S})}
\to \mathcal{O}_{A_{\begin{bf} a \end{bf}},\mathcal{S}}\cong (A_{\begin{bf} a
\end{bf}})_\mathcal{S}$$
is an isomorphism for any $\mathcal{S}\in \text{F}(A_{\begin{bf} a
\end{bf}})$.  Therefore $(f,f^{\sharp})$ is an isomorphism.
\end{proof}
Finally, we can introduce the notion of F-schemes.
\begin{definition}
An \textbf{affine F-scheme} is a ringed space which is isomorphic to
$(\text{F}(A),\mathcal{O}_{A})$
for some ring $A$. An \textbf{F-scheme} is a ringed space which is
locally isomorphic to affine F-schemes.
A \textbf{morphism} between   F-schemes is a  local morphism between them
as ringed spaces.
\end{definition}
Clearly the composition of two local morphisms is a local morphism. So F-schemes
and morphisms between them
form a category which is denoted by $\textbf{FSch}$.
An \textbf{open sub-F-scheme} of an F-scheme is an open ringed subspace of it.
By Proposition \ref{fundamentalopensubfschemes P}, an open sub-F-scheme of an F-scheme
is an F-scheme. The following lemma is straightforward to prove.
 \begin{lemma}\label{beingaffine L}
An F-scheme is affine iff it has a center. Moreover, any open affine sub-F-scheme
of an affine F-scheme, is a fundamental open subset (with its induced structures).
\end{lemma}
To see how far an F-scheme is from being an affine F-scheme, I introduce the following
definition.
\begin{definition}
An F-scheme $X$ is called $n$-affine ($n\geq 1$) if \\
1) there is an open affine  covering $U_i$ of $X$
such that each open subspace $U_i\cap U_j$ is $(n-1)$-affine (a $0$-affine
F-scheme is just an affine F-scheme), \\
2) $X$ is not $m$-affine for any $m<n$.
\end{definition}
The following lemma states that there are no $n$-affine schemes if $n>2$.
\begin{lemma}\label{naffine L}
Each F-scheme is either affine or 1-affine or 2-affine.
\end{lemma}
\begin{proof}
First I show that any open sub-F-scheme $U$ of an affine F-scheme $\text{F}(A)$
is either affine or 1-affine. In fact
$U$ can be covered with fundamental open subsets which are affine. Since the
intersection of any two fundamental open subsets in an affine F-scheme
is a fundamental open subset, we see that $U$ is either affine or 1-affine. Now
suppose that $X$ is an F-scheme and $U_i$ is an open affine covering of $X$.
Then each $U_i\cap U_j$, being  an open subset of an
affine F-scheme, is either
affine or 1-affine. So $X$ is either affine or 1-affine or 2-affine.
\end{proof}
The following proposition is the counterpart of a well known result in algebraic geometry.
\begin{proposition}\label{mapstoaffineF-schemes P}
For any F-scheme $X$ and any  affine F-scheme  $ \text{F}(A)$, the natural map
$$ \text{Hom}_{\textbf{FSch}}(X,\text{F}(A))\to \text{Hom}_{\textbf{Rings}}
(A,\Gamma(X,\mathcal{O}_X))$$
is a bijection.
\end{proposition}
\begin{proof}
If $X$ is affine, it follows from Theorem \ref{AffineF-schemes T}.
 It is easy to see that if the above statement is true for
an open  covering of $X$ and the intersection of open subsets of this covering,
then it holds for $X$ (because the
isomorphism is natural). So the statement follows from
Lemma \ref{naffine L}.
\end{proof}
Concerning morphisms from affine F-schemes to F-schemes, we have the following proposition.
 \begin{proposition}\label{morphismsfromaffineF-scheme P}
Let $X$   be an F-scheme and $\text{F}(A)$ be an affine scheme. Giving a morphism
$\text{F}(A)\to X$ is the same
as giving a point $x\in X$ and a local homomorphism $\mathcal{O}_{X,x}\to A$.
\end{proposition}
\begin{proof}
For any morphism $f: \text{F}(A)\to X$, we obtain a point $x=f(\text{U}(A))$ and a
local homomorphism
$f^\sharp_{\text{U}(A)}: \mathcal{O}_{X,x}\to A=\mathcal{O}_{A,\text{U}(A)}$.
Conversely, given such  data, we have
  a morphism  $\text{F}(A)\to \text{F}(\mathcal{O}_{X,x})$ which composed with
  the natural homomorphism
$\text{F}(\mathcal{O}_{X,x})\to X$, gives a morphism  $\text{F}(A)\to X$
sending $\text{U}(A)$ to $x$.
\end{proof}
\begin{remark}
In \cite{Va}, the author constructs a ringed space  $\text{NCSpec}(A)$ for any ring $A$.
Even though there are some relations between ringed spaces $\text{NCSpec}(A)$ and
$\text{F}(A)$, they are
not the same in general.   In fact,   $\text{NCSpec}(A)$ is, roughly speaking, the same as
the set of fully invertible \underline{subsets} of $A$ which is in general
different from $\text{F}(A)$.
\end{remark}
Finally, I give some examples of F-schemes.
\begin{example}
Clearly an F-scheme contains only one point iff it
is the full spectrum of a self-localized ring. So, for example,
$\text{F}(\text{M}_n(k))$ (where $k$ is a commutative field) is an F-scheme with only one point
$\text{U}(\text{M}_n(k))$.
\end{example}
To give the next example, we need a construction regarding ringed spaces.
Given two ringed spaces $(X,\mathcal{O}_X)$ and $(Y,\mathcal{O}_Y)$,
the topological product
of them is defined to be  the ringed space $(X\times_{\textbf{Top}} Y,
\mathcal{O}_X\times \mathcal{O}_Y)$
where $X\times_{\textbf{Top}} Y$ is the product of $X$ and $Y$
in the category \textbf{Top} of topological spaces  and the structure sheaf
$\mathcal{O}_X\times \mathcal{O}_Y$ is defined by
$\mathcal{O}_X\times \mathcal{O}_Y(U\times V):=\mathcal{O}_X(U)\times \mathcal{O}_Y(V)$
(the direct product of rings)
for any open subsets $U\subset X$ and $V\subset Y$.

\begin{example}
Suppose that $A\times B$ is the direct product of rings $A$ and $B$.
Then   $\text{F}(A\times B)$ is a disjoint union of two open subsets
$D(1_A\times 0)$, $D(0\times 1_B)$ and the closed subset
$Y=\text{F}(A\times B)\setminus (D(1_A\times 0)\cup D(0\times 1_B))$.
It is easy to see that  the projection homomorphisms $A\times B\to A$
and $A\times B\to B$ give rise to isomorphisms
$D(1_A\times 0)\cong \text{F}(A)$ and
 $D(0\times 1_B)\cong\text{F}(B)$. Moreover, the ringed subspace $Y$ of
$\text{F}(A\times B)$ is isomorphic to the topological product of  $\text{F}(A)$ and
$\text{F}(B)$.

\end{example}

\begin{example}
When $A$ is a commutative ring,   it turns out that $\text{F}(A)$ can be constructed
from
$\text{Spec}(A)$, see the next section. Nevertheless, to motivate the next example,
let's consider $\text{F}(k[[x]])$ where $k$ is a commutative field and
$k[[x]]$ is the ring of formal power series. It is easy to see that
 $\text{F}(k[[x]])=\text{Spec}(k[[x]])$ (as ringed spaces), hence $\text{F}(k[[x]])$
consists of only two points
 namely the closed point  $\text{U}(k[[x]])$ and the generic point $k[[x]]\setminus \{0\}$.
 Moreover the stalk of the structure sheaf at the closed point is $k[[x]]$ and
 its stalk at the generic point is $k((x))$, the filed of  Laurent series.
\end{example}
\begin{example}
Suppose that $k$ is a commutative field with a derivation, i.e. a linear map
$':k\to k$ such that $(ab)'=a'b+ab'$ for any $a,b\in k$, see
\cite{Ka1} for generalities on differential algebras.
One can associate a few  "differential rings" to $k$, see \cite{Ar}. One of these rings is
the ring of Volterra operators $k[[\partial^{-1}]]$ which is defined to be the set of
formal sums
$a_0+a_1\partial^{-1}+a_2\partial^{-2}+\cdots$  where $a_i \in k$ and the product
is defined by   $\partial^{-1} a=\sum_{i=0}^{\infty}{(-1)^{i} a^{(i)}\partial^{-1-i}}$
(with the convention $a^{(0)}=a$ and $a^{(n+1)}=(a^{(n)})'$).  One can see that
the affine F-scheme $\text{F}(k[[\partial^{-1}]])$ consists of only two points, namely the
closed point  $\text{U}(k[[\partial^{-1}]])$ and the generic point
$k[[\partial^{-1}]]\setminus \{0\}$.
The stalk of the structure sheaf at the closed point is just $k[[\partial^{-1}]]$
and its stalk at the generic point is $k((\partial^{-1}))$, the field
of formal pseudo differential
operators, see \cite{Ar}.
\end{example}
\begin{example}\label{nonseparatdexample E}
Given F-schemes $X$ and $Y$, open subsets $U\subset X$ and $V\subset Y$,
and an isomorphism $U\cong V$, one can glue $X$ and $Y$ by identifying $U$ and $V$, see
\cite{Ha} for gluing schemes. As an example, suppose that $X=\text{F}(A)$
is an affine F-scheme. Then $U=X\setminus\{\text{U}(A)\}$ is an open subset of $X$
and using the identity map $U\to U$, one can glue
$X$ and itself along $U$ to obtain an F-scheme which is certainly not an affine F-scheme,
because it does not have a center.
\end{example}
\end{subsection}

\end{section}

%%%%%%%%%%%%%%%%%%%%%%%%%%%%%%%%%%%%%%%%%%%%%%%%%%%%%%%%%%%%%%%%%%%%%%%%%%%%%%%%%%%%%%%%%%%%%%%%
%%%%%%%%%%%%%%%%%%%%%%%%%%%%%%%%%%%%%%%%%%%%%%%%%%%%%%%%%%%%%%%%%%%%%%%%%%%%%%%%%%%%%%%%%%%%%%%%
%%%%%%%%%%%%%%%%%%%%%%%%%%%%%%%%%%%%%%%%%%%%%%%%%%%%%%%%%%%%%%%%%%%%%%%%%%%%%%%%%%%%%%%%%%%%%%%%
%%%%%%%%%%%%%%%%%%%%%%%%%%%%%%%%%%%%%%%%%%%%%%%%%%%%%%%%%%%%%%%%%%%%%%%%%%%%%%%%%%%%%%%%%%%%%%%%
%%%%%%%%%%%%%%%%%%%%%%%%%%%%%%%%%%%%%%%%%%%%%%%%%%%%%%%%%%%%%%%%%%%%%%%%%%%%%%%%%%%%%%%%%%%%%%%%
%%%%%%%%%%%%%%%%%%%%%%%%%%%%%%%%%%%%%%%%%%%%%%%%%%%%%%%%%%%%%%%%%%%%%%%%%%%%%%%%%%%%%%%%%%%%%%%%
%%%%%%%%%%%%%%%%%%%%%%%%%%%%%%%%%%%%%%%%%%%%%%%%%%%%%%%%%%%%%%%%%%%%%%%%%%%%%%%%%%%%%%%%%%%%%%%%
%%%%%%%%%%%%%%%%%%%%%%%%%%%%%%%%%%%%%%%%%%%%%%%%%%%%%%%%%%%%%%%%%%%%%%%%%%%%%%%%%%%%%%%%%%%%%%%%
%%%%%%%%%%%%%%%%%%%%%%%%%%%%%%%%%%%%%%%%%%%%%%%%%%%%%%%%%%%%%%%%%%%%%%%%%%%%%%%%%%%%%%%%%%%%%%%%
%%%%%%%%%%%%%%%%%%%%%%%%%%%%%%%%%%%%%%%%%%%%%%%%%%%%%%%%%%%%%%%%%%%%%%%%%%%%%%%%%%%%%%%%%%%%%%%%
%%%%%%%%%%%%%%%%%%%%%%%%%%%%%%%%%%%%%%%%%%%%%%%%%%%%%%%%%%%%%%%%%%%%%%%%%%%%%%%%%%%%%%%%%%%%%%%%
%%%%%%%%%%%%%%%%%%%%%%%%%%%%%%%%%%%%%%%%%%%%%%%%%%%%%%%%%%%%%%%%%%%%%%%%%%%%%%%%%%%%%%%%%%%%%%%%
%%%%%%%%%%%%%%%%%%%%%%%%%%%%%%%%%%%%%%%%%%%%%%%%%%%%%%%%%%%%%%%%%%%%%%%%%%%%%%%%%%%%%%%%%%%%%%%%
%%%%%%%%%%%%%%%%%%%%%%%%%%%%%%%%%%%%%%%%%%%%%%%%%%%%%%%%%%%%%%%%%%%%%%%%%%%%%%%%%%%%%%%%%%%%%%%%
%%%%%%%%%%%%%%%%%%%%%%%%%%%%%%%%%%%%%%%%%%%%%%%%%%%%%%%%%%%%%%%%%%%%%%%%%%%%%%%%%%%%%%%%%%%%%%%%
%%%%%%%%%%%%%%%%%%%%%%%%%%%%%%%%%%%%%%%%%%%%%%%%%%%%%%%%%%%%%%%%%%%%%%%%%%%%%%%%%%%%%%%%%%%%%%%%

\begin{section}{Schemes versus F-schemes}
In this part a connection between schemes and F-schemes is discussed.
In particular, the notion
of ``commutative'' F-schemes is defined and it is shown that commutative F-schemes
and schemes are closely related.
If $A$ is commutative then $\text{F}(A)$ and $\text{NCSpec}(A)\setminus \{\gamma\}$
are the same where
$\gamma$ is the generic point of $\text{NCSpec}(A)$, see \cite{Va}.
For another point of view on the relation between
$\text{F}(A)=\text{NCSpec}(A)\setminus \{\gamma\}$ and $\text{Spec}(A)$, see \cite{Va}.

%%%%%%%%%%%%%%%%%%%%%%%%%%%%%%%%%%%%%%%%%%%%%%%%%%%%%%%%%%%%%%%%%%%%%%%%%%%%%%%%%%%%%%%%%%%%%%%%
%%%%%%%%%%%%%%%%%%%%%%%%%%%%%%%%%%%%%%%%%%%%%%%%%%%%%%%%%%%%%%%%%%%%%%%%%%%%%%%%%%%%%%%%%%%%%%%%
%%%%%%%%%%%%%%%%%%%%%%%%%%%%%%%%%%%%%%%%%%%%%%%%%%%%%%%%%%%%%%%%%%%%%%%%%%%%%%%%%%%%%%%%%%%%%%%%
%%%%%%%%%%%%%%%%%%%%%%%%%%%%%%%%%%%%%%%%%%%%%%%%%%%%%%%%%%%%%%%%%%%%%%%%%%%%%%%%%%%%%%%%%%%%%%%%

\begin{subsection}{Fully invertible systems in the commutative case}
In this section, some standard facts from the theory of
commutative rings are reviewed, see \cite{AM,Ka} for example.

Suppose that $A$ is a unitary
commutative ring.
A subset $S$ of $A$ is  called \textbf{multiplicatively closed} if it
is closed under multiplication.   Given a multiplicatively closed subset $S$
of $A$, we have $S^{-1}A$, the ring of fractions of $A$ with
respect to $S$ and a natural map $A\to S^{-1}A$, sending $a\in A$
to $a/1$.
If $P$ is a prime ideal of $A$ then $A\setminus P$ is a
multiplicatively closed subset whose ring of
fractions is denoted by $A_P$. We need the following lemma to characterize fully
invertible subsets in commutative rings.
\begin{lemma}\label{invertible L}
Suppose that $S$ is a multiplicatively closed subset  of $A$.
Then an element $a\in A$ is invertible in $S^{-1}A$ if and only
if $Aa\cap S\neq \emptyset$.
\end{lemma}
\begin{proof}
If $Aa\cap S\neq \emptyset$then $ab=s\in S$ for some $b\in A$.
This implies that $a.(b/s)=1$ in $S^{-1}A$, i.e. $a\in A$ is
invertible in $S^{-1}A$. Conversely if $a\in A$ is invertible in
$S^{-1}A$ then  we have $a.(b/s)=1$ in $S^{-1}A$ for some $b\in A$
and $s\in S$. This implies that $s'(ab-s)=0$ in $A$, for some
$s'\in S$. So $(s'b)a=s's\in Aa\cap S$, i.e. $Aa\cap S\neq
\emptyset$.
\end{proof}
\begin{proposition}
A subset $S$ of $A$ is fully invertible iff the following hold
\begin{enumerate}
\item $S$ does not contain the zero element.
\item $S$ is multiplicatively closed.
\item $ab\in S$ implies that $a\in S$ and $b\in S$.
\end{enumerate}
\end{proposition}
\begin{proof}
It is easy to see that if $S$ is fully invertible then it satisfies
1,2 and 3. Conversely, since $S$ is multiplicatively closed,
we have an isomorphism $A_S\cong S^{-1}A$. Then using Lemma
\ref{invertible L}, we see that $i^{-1}(\text{U}(S^{-1}A))=S$. Note that
$S^{-1}A=0$ iff $0\in S$.
\end{proof}
\begin{remark}
In \cite{Ka}, the expression ``saturated multiplicatively closed'' is used
instead of ``fully invertible''.
\end{remark}
It is easy to see that  for any prime ideal $P$ of $A$, the subset
$A\setminus P$ is a fully invertible subset of $A$.
Therefore, as sets, we have an inclusion
$$\text{Spec}(A)\subset \text{F}(A).$$
The following
proposition characterizes such fully invertible subsets.
\begin{proposition}\label{primefully invertible P}
A fully invertible subset $S$ of $A$ is of the form $A\setminus P$ for
some prime ideal $P$ of $A$,  iff $S^{-1}A$ is a local ring.
\end{proposition}
\begin{proof}
If $P$ is a prime ideal then it is well-known that $(A\setminus P)^{-1}A=A_P$ is a
local ring.
Conversely, if $S^{-1}A$ is a local ring with the maximal ideal $m$ then  $P=m\cap A$
gives a prime ideal
of $A$. Moreover $S=\text{U}(S^{-1}A)\cap A=A\setminus P$.
\end{proof}
Note that for any $a\in A$,
the intersection of a fundamental open subset $D(a)$ of $\text{F}(A)$ with $\text{Spec}(A)$
is just the   open subset
$$D(a)\cap\text{Spec}(A)= \{P\in\text{Spec}(A)|\, a\notin P\}$$
of $\text{Spec}(A)$ associated to $a$, see \cite{Ha}.

It is well known that for any multiplicatively closed subset $S$
of $A$, we have a one-to-one correspondence between prime ideals
of $S^{-1}A$ and those prime ideals $P$ of $A$  for which $P\cap
S=\emptyset$. In other words
\begin{equation}\label{closed}
Spec(S^{-1}A)=\{P\in Spec(A)|\, P\cap S=\emptyset\, \}
\end{equation}
In the case, $S$ is fully invertible, we can get more information. Namely
\begin{proposition}\label{equal}
Two fully invertible subsets $S$ and $T$ are the same iff
$\text{Spec}(S^{-1}A)=\text{Spec}(T^{-1}A)$ under the above identification.
Moreover we have
$$S=\cap_{P\in Spec(S^{-1}A)}(A\setminus P)$$
for
any fully invertible subset $S$.
\end{proposition}
\begin{proof}
Suppose that $Spec(S^{-1}A)=Spec(T^{-1}A)$, but there is some $a\in S\setminus T$.
Then there is a prime ideal
$P$ such that $a\in P$ but $P\cap T=\emptyset$, i.e. $P\in Spec(T^{-1}A)$ but
$P\notin Spec(S^{-1}A)$, a contradiction.\\
Clearly $S\subset\cap_{P\in Spec(S^{-1}A)}(A\setminus P)$. Now given any $a\in A\setminus S$,
there is a prime ideal
$P$ such that $a\in P$ but $P\cap S=\emptyset$. This means that $a\notin A\setminus P$ and so
 $\cap_{P\in Spec(S^{-1}A)}(A\setminus P)\subset S$ as well.
\end{proof}
Note that this proposition implies that any fully invertible subset of
$A$ is of the form $S=\cap_{P\in I}(A\setminus P)$, where $I$ is a
family of prime ideals of $A$ (also see  \cite{Ka}).

It is easy to see that, in the commutative case, the notions of fully invertible subsets and
fully invertible systems coincide. Therefore if $A$ is a commutative ring then,
$\text{F}(A)$  is just the set of fully invertible subsets of $A$.

\end{subsection}

%%%%%%%%%%%%%%%%%%%%%%%%%%%%%%%%%%%%%%%%%%%%%%%%%%%%%%%%%%%%%%%%%%%%%%%%%%%%%%%%%%%%%%%%%%%%%%%%
%%%%%%%%%%%%%%%%%%%%%%%%%%%%%%%%%%%%%%%%%%%%%%%%%%%%%%%%%%%%%%%%%%%%%%%%%%%%%%%%%%%%%%%%%%%%%%%%
%%%%%%%%%%%%%%%%%%%%%%%%%%%%%%%%%%%%%%%%%%%%%%%%%%%%%%%%%%%%%%%%%%%%%%%%%%%%%%%%%%%%%%%%%%%%%%%%
%%%%%%%%%%%%%%%%%%%%%%%%%%%%%%%%%%%%%%%%%%%%%%%%%%%%%%%%%%%%%%%%%%%%%%%%%%%%%%%%%%%%%%%%%%%%%%%%

\begin{subsection}{Commutative F-schemes}
The way schemes and F-schemes are defined shows that there is a relation
between schemes and F-schemes.
In this section, a functor is given from the category of ``commutative''
F-schemes to the category of schemes.
\begin{definition}
An $F$-scheme $X$ is called commutative if there is an open
  affine covering $  \{\text{F}(A_i)\}_i$ of $X$ such that $A_i's$ are
commutative rings.
\end{definition}
The following lemma gives a characterization of F-schemes.
\begin{lemma}
An F-scheme $X$ is commutative iff for any open affine sub-F-scheme
$\text{F}(A)$ of $X$, $A$ is a commutative ring.
\end{lemma}
\begin{proof}
One direction is trivial. Suppose that $X$ is commutative and $\text{F}(A)$
is an open affine sub-F-scheme of $X$. Since $X$ is commutative, for any $x\in X$,
the ring $\mathcal{O}_{X,x}$ is
commutative. So $A\cong\mathcal{O}_{X,\text{U}(A)}$ is commutative.
\end{proof}
\begin{definition}
Let $X$ be a ringed space. A point $x\in X$ is called \textbf{local} if
$\mathcal{O}_{X,x}$ is a local ring.
Recall that a local ring is a ring whose set of  noninvertible elements is an ideal.
\end{definition}
Denote  the set of local points of $X$ by $\text{L}(X)$ and consider it as a ringed
subspace of
$X$ with its induced structures.
First we have the following simple algebraic lemma.
\begin{lemma}\label{simplealgebaraic L}
Suppose that $f:A\to B$ is a local ring homomorphism and $B$ is a local ring.
Then $A$ is a local ring.
\end{lemma}
\begin{proof}
Let $m_B$ be the unique maximal ideal of $B$. Since $f$ is local, any element
not in $f^{-1}(m_B)$ is invertible, i.e.
$A$ is a local ring with the maximal ideal $f^{-1}(m_B)$.
\end{proof}
Using this lemma we have the following proposition.
\begin{proposition}\label{local P}
Let $X$ and $Y$ be two ringed spaces and $f:X\to Y$ is a local morphism between
them. Then the restriction of
$f$  to $\text{L}(X)$, defines a local  morphism   $\text{L}(f):\text{L}(X)\to \text{L}(Y)$.
\end{proposition}
\begin{proof}
From Lemma \ref{simplealgebaraic L}, we deduce that $f(\text{L}(X))\subset \text{L}(Y)$.
Since the structures of
$\text{L}(X)$ and $\text{L}(Y)$ are the ones induced from $X$ and $Y$, we see that $f$
restricted to $\text{L}(X)$, gives a local morphism
$\text{L}(f):\text{L}(X)\to \text{L}(Y)$.
\end{proof}
\begin{lemma}\label{Affine L}
For any (commutative) ring $A$, there is a natural local morphism of
ringed spaces $f=f_A:\text{Spec}(A)\to \text{F}(A)$.
Moreover, $f$ is a homeomorphism of $\text{Spec}(A)$ onto $\text{L}(\text{F}(A))$
and the structure sheaf of
$\text{Spec}(A)$ is the induced
structure sheaf from $\text{F}(A)$, i.e. the ringed spaces $\text{Spec}(A)$
and $\text{L}(\text{F}(A))$
are isomorphic.
\end{lemma}
\begin{proof}
Define $f$ by $f(P)=A\setminus P$ for any prime ideal $P$
of $A$. It is easy to see that $f$ defines a topological homeomorphism
from $ \text{Spec}(A)$
onto $\text{L}(\text{F}(A))$, see Proposition    \ref{primefully invertible P}.
The way the structure sheaves of
$\text{Spec}(A)$ and $ \text{F}(A) $ are defined shows that there is an
obvious morphism $f^{\sharp}$, yielding
a morphism from  $\text{Spec}(A)$ onto     $\text{L}(\text{F}(A))$ with the
induced structures. This map is easily seen to have
the desired properties.
\end{proof}
By the above lemma, we can naturally consider $\text{Spec}(A)$ as a
ringed subspace of $\text{F}(A)$. Finally we have the following theorem.
\begin{theorem}\label{equivalenceofcategories T}
For any commutative scheme $X$, consider $\text{L}(X)$ with the induced ringed
structure from
$X$. Then the assignment $X\mapsto \text{L}(X)$ defines a functor from the
category of commutative F-schemes
to the category of  schemes.
\end{theorem}
\begin{proof}
From Lemma \ref{Affine L}, we can see that $\text{L}(X)$ is in fact a scheme.
Therefore the assignment does assign
a scheme to any commutative F-scheme. Suppose that $f:X\to Y$ is a
morphism between two commutative F-schemes. Since $f$ is local,
$f$ maps $\text{L}(X)$ to $\text{L}(Y)$ and we obtain a local morphism
$\text{L}(f):\text{L}(X)\to \text{L}(Y)$,
by Proposition \ref{local P}.  So  $\text{L}(f):\text{L}(X)\to \text{L}(Y)$ is a
morphism of schemes.
Now it is easy to see that
this gives a functor
$$\text{L}:\textbf{CFSch}\to \textbf{Sch}$$
from the category of commutative F-schemes to the category of schemes.
\end{proof}

\end{subsection}

\end{section}

%%%%%%%%%%%%%%%%%%%%%%%%%%%%%%%%%%%%%%%%%%%%%%%%%%%%%%%%%%%%%%%%%%%%%%%%%%%%%%%%%%%%%%%%%%%%%%%%
%%%%%%%%%%%%%%%%%%%%%%%%%%%%%%%%%%%%%%%%%%%%%%%%%%%%%%%%%%%%%%%%%%%%%%%%%%%%%%%%%%%%%%%%%%%%%%%%
%%%%%%%%%%%%%%%%%%%%%%%%%%%%%%%%%%%%%%%%%%%%%%%%%%%%%%%%%%%%%%%%%%%%%%%%%%%%%%%%%%%%%%%%%%%%%%%%
%%%%%%%%%%%%%%%%%%%%%%%%%%%%%%%%%%%%%%%%%%%%%%%%%%%%%%%%%%%%%%%%%%%%%%%%%%%%%%%%%%%%%%%%%%%%%%%%
%%%%%%%%%%%%%%%%%%%%%%%%%%%%%%%%%%%%%%%%%%%%%%%%%%%%%%%%%%%%%%%%%%%%%%%%%%%%%%%%%%%%%%%%%%%%%%%%
%%%%%%%%%%%%%%%%%%%%%%%%%%%%%%%%%%%%%%%%%%%%%%%%%%%%%%%%%%%%%%%%%%%%%%%%%%%%%%%%%%%%%%%%%%%%%%%%
%%%%%%%%%%%%%%%%%%%%%%%%%%%%%%%%%%%%%%%%%%%%%%%%%%%%%%%%%%%%%%%%%%%%%%%%%%%%%%%%%%%%%%%%%%%%%%%%
%%%%%%%%%%%%%%%%%%%%%%%%%%%%%%%%%%%%%%%%%%%%%%%%%%%%%%%%%%%%%%%%%%%%%%%%%%%%%%%%%%%%%%%%%%%%%%%%
%%%%%%%%%%%%%%%%%%%%%%%%%%%%%%%%%%%%%%%%%%%%%%%%%%%%%%%%%%%%%%%%%%%%%%%%%%%%%%%%%%%%%%%%%%%%%%%%
%%%%%%%%%%%%%%%%%%%%%%%%%%%%%%%%%%%%%%%%%%%%%%%%%%%%%%%%%%%%%%%%%%%%%%%%%%%%%%%%%%%%%%%%%%%%%%%%
%%%%%%%%%%%%%%%%%%%%%%%%%%%%%%%%%%%%%%%%%%%%%%%%%%%%%%%%%%%%%%%%%%%%%%%%%%%%%%%%%%%%%%%%%%%%%%%%
%%%%%%%%%%%%%%%%%%%%%%%%%%%%%%%%%%%%%%%%%%%%%%%%%%%%%%%%%%%%%%%%%%%%%%%%%%%%%%%%%%%%%%%%%%%%%%%%
 %%%%%%%%%%%%%%%%%%%%%%%%%%%%%%%%%%%%%%%%%%%%%%%%%%%%%%%%%%%%%%%%%%%%%%%%%%%%%%%%%%%%%%%%%%%%%%%%
%%%%%%%%%%%%%%%%%%%%%%%%%%%%%%%%%%%%%%%%%%%%%%%%%%%%%%%%%%%%%%%%%%%%%%%%%%%%%%%%%%%%%%%%%%%%%%%%

\begin{section}{First properties of F-schemes}
In this section some familiar concepts in the
theory of schemes such as closed sub-schemes and separated schemes (see \cite{Ha}),
are extended to F-schemes.

%%%%%%%%%%%%%%%%%%%%%%%%%%%%%%%%%%%%%%%%%%%%%%%%%%%%%%%%%%%%%%%%%%%%%%%%%%%%%%%%%%%%%%%%%%%%%%%%
%%%%%%%%%%%%%%%%%%%%%%%%%%%%%%%%%%%%%%%%%%%%%%%%%%%%%%%%%%%%%%%%%%%%%%%%%%%%%%%%%%%%%%%%%%%%%%%%
%%%%%%%%%%%%%%%%%%%%%%%%%%%%%%%%%%%%%%%%%%%%%%%%%%%%%%%%%%%%%%%%%%%%%%%%%%%%%%%%%%%%%%%%%%%%%%%%
%%%%%%%%%%%%%%%%%%%%%%%%%%%%%%%%%%%%%%%%%%%%%%%%%%%%%%%%%%%%%%%%%%%%%%%%%%%%%%%%%%%%%%%%%%%%%%%%
\begin{subsection}{Quasi-coherent sheaves of modules}
Let $A$ be a ring. Suppose that $M$ is a left $A$-module and $N$ is an $A$-bimodule. Using
a similar construction as in Remark \ref{definitionofstructuresheaf R}, we obtain  a sheaf of
$\mathcal{O}_A$-modules $\widetilde{M}$ on
$\text{F}(A)$ such that
$\widetilde{M}(D(\begin{bf} a \end{bf}))=A_{\begin{bf} a \end{bf}}\otimes_A M $,
for any fundamental open subset
$D(\begin{bf} a \end{bf})$ of $\text{F}(A)$. Similarly, we obtain a sheaf of
$\mathcal{O}_A$-bimodules $\widehat{N}$ on
$\text{F}(A)$ such that
$\widehat{N}(D(\begin{bf} a \end{bf}))=A_{\begin{bf} a
\end{bf}}\otimes_A N\otimes_A A_{\begin{bf} a \end{bf}} $, for any fundamental
open subset
$D(\begin{bf} a \end{bf})$ of $\text{F}(A)$. Note that for any open subset $U$
of $\text{F}(A)$, we have
$$\widetilde{M}(U)=\underset{D(\begin{bf} a \end{bf})\subset U}
{\underleftarrow{\lim}} (A_{\begin{bf} a \end{bf}}\otimes_A M)$$
and
$$\widehat{N}(U)=\underset{D(\begin{bf} a \end{bf})\subset U}
{\underleftarrow{\lim}} (A_{\begin{bf} a \end{bf}}\otimes_A N \otimes
A_{\begin{bf} a \end{bf}}).$$
Clearly $\Gamma(\text{F}(A), \widetilde{M})=M$ and $\Gamma(\text{F}(A), \widehat{N})=N$.
\begin{proposition}\label{stalkofsheafofmodules P}
\begin{enumerate}
\item For each $\mathcal{S}\in \text{F}(A)$, there is a natural $A_{\mathcal{S}}$-module
 isomorphism
 $\widetilde{M}_{\mathcal{S}}\cong A_{\mathcal{S}}\otimes_A M$.
\item For any fraction $\begin{bf} a \end{bf}$ of $A$ such that
$D(\begin{bf} a \end{bf})\neq \emptyset$, there is an isomorphism
of sheaves of $\mathcal{O}_X$-modules
$\widetilde{M}|_{D(\begin{bf} a \end{bf})}\cong \widetilde{M_{\begin{bf} a \end{bf}}}$
where $M_{\begin{bf} a \end{bf}}=A_{\begin{bf} a \end{bf}}\otimes_A M$.
\end{enumerate}
Similar statements hold for  $\widehat{N}$.
\end{proposition}
\begin{proof}
We have
$$\widetilde{M}_{\mathcal{S}}\cong \underset{\mathcal{S}\in U}
{\underrightarrow{\lim}} \widetilde{M}(U)
\cong \underset{\mathcal{S}\in D(\begin{bf} a \end{bf})}{\underrightarrow{\lim}}
 \widetilde{M}(D(\begin{bf} a \end{bf}))$$
$$
\cong \underset{\mathcal{S}\in D(\begin{bf} a \end{bf})}{\underrightarrow{\lim}}
(A_{\begin{bf} a \end{bf}}\otimes_A M)
\cong (\underset{\mathcal{S}\in D(\begin{bf} a \end{bf})}{\underrightarrow{\lim}}
A_{\begin{bf} a \end{bf}})\otimes_A M
\cong A_{\mathcal{S}}\otimes_A M
$$
The second assertion is trivial.
\end{proof}
The following proposition is straightforward to prove.
\begin{proposition}
The assignment $M\mapsto \widetilde{M}$ defines a right exact functor from the category
of $A$-modules  to the category of left $\mathcal{O}_A$-modules.
Moreover the canonical map $\text{Hom}_A(M,N)\to
\text{Hom}_{\mathcal{O}_A}(\widetilde{M},\widetilde{N})$ is an isomorphism for any
two $A$-modules $M$ and $N$. One has similar statements  for bimodules.

\end{proposition}
Now we can introduce the notion of quasi-coherent  and coherent sheaves of modules.
\begin{definition}
Let $X$ be an F-scheme. Suppose that $\mathcal{F}$ is a sheaf of $\mathcal{O}_X$-modules and
$\mathcal{G}$ is a sheaf of $\mathcal{O}_X$-bimodules.  The sheaf
$\mathcal{F}$ (resp. $\mathcal{G}$) is called \textbf{quasi-coherent}
if there is an open affine covering $U_i=\text{F}(A_i)$ of $X$ and
$A_i$-modules $M_i$ (resp. $A_i$-bimodules $N_i$) such that
$\mathcal{F}|_{U_i}\cong \widetilde{M}_i$ (resp. $\mathcal{G}|_{U_i}\cong \widehat{N_i}$).
They  called \textbf{coherent} if  , furthermore, $M_i$'s (resp. $N_i$'s) can be taken to be
finitely generated
$A_i$-modules ($A_i$-bimodules).
\end{definition}
The following proposition gives a familiar property of quasi-coherent (and coherent)
sheaves of modules in the case of F-schemes.
\begin{proposition}
Suppose that $X$ is an F-scheme and   $\mathcal{F}$ is a sheaf of $\mathcal{O}_X$-modules. Then
$\mathcal{F}$ is quasi-coherent (coherent) iff for any open affine
 sub-F-scheme $U=\text{F}(A)$ of $X$,
there is an (finitely generated) $A$-module $M$ such that $\mathcal{F}|_{U}\cong \widetilde{M}$.
 A similar statement
holds for sheaves of bimodules.
\end{proposition}
\begin{proof}
One direction is trivial. So suppose that
$\mathcal{F}$ is quasi-coherent and $U=\text{F}(A)$ is an open affine
sub-F-scheme  of $X$. Since $\text{F}(A)$ has a center,
there is an open affine
sub-F-scheme $U\subset V=\text{F}(B)$ and a $B$-module $M$ such that
$\mathcal{F}|_{V}\cong \widetilde{M}$. So, by Lemma \ref{beingaffine L},
 $U$ is a fundamental open subset of
$V$ and hence $A\cong B_{\begin{bf} a \end{bf}}$ for some fraction of $B$. This implies that
$\mathcal{F}|_U\cong \widetilde{M}|_U\cong \widetilde{M_{\begin{bf} a \end{bf}}}$,
by Proposition
\ref{stalkofsheafofmodules P}. If $\mathcal{F}$ is  coherent then  $M$ can be
chosen to be a finitely generated $B$-module. So
$M_{\begin{bf} a \end{bf}}=B_{\begin{bf} a \end{bf}}\otimes_B M\cong A\otimes_B M$ is a
finitely generated $A$-modules.
\end{proof}

Note that we have similar notions and constructions for
right modules as well.

\begin{proposition}
Suppose that $X$ is an affine F-scheme and
$$0\to \mathcal{F}_1\to\mathcal{F}_2\to\mathcal{F}_3\to 0$$
is an exact sequence of abelian sheaves on $X$. Then the sequence
$$0\to \Gamma(X,\mathcal{F}_1)\to\Gamma(X,\mathcal{F}_2)\to\Gamma(X,\mathcal{F}_3)\to 0$$
is exact.
\end{proposition}
\begin{proof}
In fact this holds for any topological space with a center, because
the group of the global sections of any abelain sheaf on a topological space with center $x$
is the same as the stalk
of that sheaf at $x$.
\end{proof}

If $M$ is a left $A$-module as well as a right $A$-module then to distinguish between
the corresponding sheave of left $\mathcal{O}_A$-modules and
sheaf of right $\mathcal{O}_A$-modules on $\text{F}(A)$, the notations $\widetilde{M}^l$ and
$\widetilde{M}^r$ are used respectively. If there is no confusion,
I do not use the indices $l$ and $r$.

\end{subsection}

%%%%%%%%%%%%%%%%%%%%%%%%%%%%%%%%%%%%%%%%%%%%%%%%%%%%%%%%%%%%%%%%%%%%%%%%%%%%%%%%%%%%%%%%%%%%%%%%
%%%%%%%%%%%%%%%%%%%%%%%%%%%%%%%%%%%%%%%%%%%%%%%%%%%%%%%%%%%%%%%%%%%%%%%%%%%%%%%%%%%%%%%%%%%%%%%%
%%%%%%%%%%%%%%%%%%%%%%%%%%%%%%%%%%%%%%%%%%%%%%%%%%%%%%%%%%%%%%%%%%%%%%%%%%%%%%%%%%%%%%%%%%%%%%%%
%%%%%%%%%%%%%%%%%%%%%%%%%%%%%%%%%%%%%%%%%%%%%%%%%%%%%%%%%%%%%%%%%%%%%%%%%%%%%%%%%%%%%%%%%%%%%%%%

\begin{subsection}{Closed sub-F-schemes}
In order to introduce the notion of closed sub-F-schemes and closed immersions,
we need some information in the affine case.
\begin{lemma}\label{closedsubsetsofideals}
Let   $A$ be a ring and $I\subset A$ be an ideal. Let $\pi:A\to A/I$ be
the quotient homomorphism inducing the morphism $f:\text{F}(A/I)\to \text{F}(A)$.
Then $\mathcal{S}\in \text{F}(A)$ is in the image of $f$ iff
 $I\subset \text{J}(A_\mathcal{S})$ (meaning that the image of $I$ under the natural homomorphism
 $A\to A_\mathcal{S}$ is a subset of $\text{J}(A_\mathcal{S})$).
\end{lemma}
\begin{proof}
Suppose that $\mathcal{S}\in \text{F}(A)$. If $\mathcal{S}=f(\mathcal{T})$ then
 we have the following   commutative diagram
 \begin{displaymath}
\xymatrix{ & A \ar[d]  \ar[r]  & A_\mathcal{S} \ar[d]   \\
           & A/I   \ar[r] &  (A/I)_\mathcal{T} }
 \end{displaymath}
The homomorphism $A_\mathcal{S}\to (A/I)_\mathcal{T}$ is local and its kernel
contains $I$. Therefore, by Proposition  \ref{kernel P}, we have
$I\subset \text{J}(A_\mathcal{S})$.
Conversely, suppose that $I\subset \text{J}(A_\mathcal{S})$. Consider the following
commutative diagram
 \begin{displaymath}
\xymatrix{ & A \ar[d]  \ar[r]  & A_\mathcal{S} \ar[d]   \\
           & A/I   \ar[r] &  A_\mathcal{S}/I' }
 \end{displaymath}
where $I'$ is the ideal of $A_\mathcal{S}$ generated by $I$. Since
$I'\subset \text{J}(A_\mathcal{S})$,
the fully invertible system on $A$ coming from $A\to A_\mathcal{S}/I'$ is just $\mathcal{S}$.
But the commutativity of the above diagram implies that $f(\mathcal{T})=\mathcal{S}$ where
$\mathcal{T}$ is the fully invertible system on $A/I$ coming from $ A/I\to A_\mathcal{S}/I'$.

\end{proof}
Given an ideal $I$ of a ring $A$, set
$$Z(I):=\{\mathcal{S}\in \text{F}(A)|\, I\subset \text{J}(A_\mathcal{S})\}.$$
By the above lemma, $Z(I)$ is the image of $\text{F}(A/I)$ under the morphism
$f:\text{F}(A/I)\to \text{F}(A)$. The following lemma characterizes its closure.
\begin{lemma}\label{closureofZ(I) L}
For any ideal $I$ of $A$, we have
$$\overline{Z(I)}= \{\mathcal{S}\in \text{F}(A)|\, \text{the ideal generated by I in}\,\,
 A_\mathcal{S}\neq A_\mathcal{S}\}.$$
\end{lemma}
\begin{proof}
Let $M$ be the set on the right-hand side of this equality. Clearly $Z(I)\subset M$.
Let $\mathcal{S}\in M$ and $I'$ be the ideal generated by $I$ in $A_\mathcal{S}$.
The commutativity
of the following diagram
 \begin{displaymath}
\xymatrix{ & A \ar[d]  \ar[r]  & A_\mathcal{S} \ar[d]   \\
           & A/I   \ar[r] &  A_\mathcal{S}/I' }
 \end{displaymath}
implies that there is some $\mathcal{T}\in Z(I)$ which contains $\mathcal{S}$. So,
by  Proposition \ref{closure of a point P}, $\mathcal{S}\in \overline{Z(I)}$.
This proves that
$M\subset \overline{Z(I)}$. To finish the proof, we just need to show that $M$ is closed.
Suppose that $\mathcal{S}\notin M$. Therefore the ideal generated by $I$ in $A_\mathcal{S}$
contains
1. Using the presentation of $A_{\mathcal{S}}$, we can find a fraction
$\begin{bf} a \end{bf}\in \mathcal{S}$ of $A$ such that the ideal generated by $I$ in
$A_{\begin{bf} a \end{bf}}$ contains 1. This implies that for any $\mathcal{T}\in\text{F}(A)$
containing
$\begin{bf} a \end{bf}$, the ideal generated by $I$ in $A_{\mathcal{T}}$ contains 1. So
$D(\begin{bf} a \end{bf})\cap M=\emptyset$. This implies that $M$ is closed.
\end{proof}
After discussing the affine case, we go to the general case.
\begin{proposition}\label{supportofquasicoherent P}
Suppose that $\mathcal{I}$ is a quasi-coherent sheaf of (two-sided) ideals on an
F-scheme $X$. Then
the support of the sheaf $\mathcal{O}_X/\mathcal{I}$    is a closed subset of
$X$.
\end{proposition}
\begin{proof}
We can assume that $X=\text{F}(A)$ is affine. Therefore $\mathcal{I}=\widehat{I}$ for some
(two-sided)
ideal $I$ of $A$. It is easy to see that the support of $\mathcal{O}_X/\mathcal{I}$ consists of
all $\mathcal{S}$ such that the ideal generated by $I$ in $A_{\mathcal{S}}$ does not contain 1. So
the support of $\mathcal{O}_X/\mathcal{I}$ is equal to $\overline{Z(I)}$, see Lemma
\ref{closureofZ(I) L}.
\end{proof}
Given  a sheaf of (two-sided) ideals $\mathcal{I}$
on an F-scheme $X$, set
$$Z(\mathcal{I}):=\{x\in X|\, \mathcal{I}_x\subset \text{J}(\mathcal{O}_{X,x})\}.$$
Note that the closure of $Z(\mathcal{I})$ is the support of $\mathcal{O}_X/\mathcal{I}$, if
$\mathcal{I}$ is quasi-coherent.
\begin{definition}
A subset of an F-scheme $X$ is called \textbf{semi-closed in $X$} if it is of the form  $Z(\mathcal{I})$
for some  quasi-coherent sheaf of (two-sided) ideals on $X$.
\end{definition}
Note that if $U\subset X$ is open and $Y\subset X$ is semi-closed in $X$ then  $Y\cap U$
is semi-closed in $U$.
\begin{proposition}\label{closedsubfschemes P}
Suppose that $\mathcal{I}$ is a quasi-coherent sheaf of (two-sided) ideals on an F-scheme $X$.
Then $(Z(\mathcal{I}),\mathcal{O}_X/\mathcal{I})$ is an F-scheme.
\end{proposition}
\begin{proof}
It is enough to prove this for the case $X=\text{F}(A)$ is affine
and  $\mathcal{I}=\widehat{I}$ for some (two-sided)
ideal $I$ of $A$. I prove that  $(Z(\mathcal{I}),\mathcal{O}_X/\mathcal{I})$
is isomorphic to   the affine F-scheme $\text{F}(A/I)$. Consider the morphism
$f:\text{F}(A/I)\to \text{F}(A)$ induced by
the quotient homomorphism $\pi:A\to A/I$. By Lemma \ref{closedsubsetsofideals}, $f$ induces
a continuous map   $g:\text{F}(A/I)\to Z(\mathcal{I})$. It is clearly a bijection.
To prove that $g$ is a homeomorphism, it is enough to show that
the images of open subsets are open
in the induced topology of  $ Z(\mathcal{I})$.  So suppose that
$\begin{bf} a \end{bf}$ is a fraction of $A/I$. It is easy to see that
there is a fraction  $\begin{bf} b \end{bf}$ of $A$ such that
$\pi(\begin{bf} b \end{bf})=\begin{bf} a \end{bf}$. Then one can
 check that $g(D(\begin{bf} a \end{bf}))=D(\begin{bf} b \end{bf})\cap Z(\mathcal{I}) $.
So $g$ is a homeomorphism. On the other hand,
it is easy to see that
$\mathcal{I}$ is in the kernel of $f^\sharp: \mathcal{O}_A\to f_*\mathcal{O}_{A/I}$. Therefore
$f$ induces a morphism $g:\text{F}(A/I)\to (Z(\mathcal{I}),\mathcal{O}_A/\mathcal{I})$ of ringed
spaces.  In order to show that $g$ is an isomorphism,
it is enough to show that for any fraction $\begin{bf} a \end{bf}$ of $A$, the homomorphism
$$f^{\sharp}:\mathcal{O}(A)(D(\begin{bf} a \end{bf}))/\mathcal{I}(D(\begin{bf} a \end{bf}))\to
f_*\mathcal{O}_{A/I}(D(\begin{bf} a \end{bf}))=\mathcal{O}(A/I)(D(\pi(\begin{bf} a \end{bf})))$$
is an isomorphism (since fundamental open subsets form a base for the topology of $\text{F}(A)$).
The above homomorphism is the natural homomorphism
$$A_{\begin{bf} a \end{bf}}/I'\to (A/I)_{\pi(\begin{bf} a \end{bf})}$$
where $I'$ is the ideal of $A_{\begin{bf} a \end{bf}}$ which is generated by $I$.
I claim that
the homomorphism $A/I \to A_{\begin{bf} a \end{bf}}/I'$ has the universal property associated to
$\pi(\begin{bf} a \end{bf})$. Suppose that $\begin{bf} a \end{bf}=(a)$ where $a\in A$. Since
$a$ is invertible in $A_a$, the element $\pi(a)$ is invertible in $A_a/I'$. If $A/I\to B$
is a ring homomorphism  such that $\pi(a)$ is invertible in $B$ then there is a
natural ring homomorphism $(A/I)_{\pi(a)}\to B$ making
the following diagram commutative
 \begin{displaymath}
\xymatrix{ & A/I  \ar[dr] \ar[r]  & (A/I)_{\pi(  a  )}\ar[d]   \\
           &      & B  }
 \end{displaymath}
This homomorphism composed with the homomorphism $A_a/I'\to (A/I)_{\pi(a)}$, gives
a homomorphism $A_a/I'\to B$ making the following diagram commutative
 \begin{displaymath}
\xymatrix{ & A/I  \ar[dr] \ar[r]  & A_{  a  }/I' \ar[d]   \\
           &      &  B }
 \end{displaymath}
So the homomorphism $A/I \to A_{  a  }/I'$ has the universal property associated to
$\pi(  a  )$. A simple induction implies the same for any fraction.
Therefore in the following  commutative diagram
 \begin{displaymath}
\xymatrix{ & A/I  \ar[dr] \ar[r]  & A_{\begin{bf} a \end{bf}}/I' \ar[d]   \\
           &      &  (A/I)_{\pi(\begin{bf} a \end{bf})} }
 \end{displaymath}
the natural homomorphism
$$A_{\begin{bf} a \end{bf}}/I'\to (A/I)_{\pi(\begin{bf} a \end{bf})}$$
is an isomorphism.
\end{proof}
Now we can introduce the notion of closed sub-F-schemes and closed immersions. First, it is
straightforward to define the notion of open
immersions; a morphism $f:X\to Y$ is called an \textbf{open immersion}
if $f$ induces an isomorphism
of $X$ with an open sub-F-scheme of $Y$.
\begin{definition}
A \textbf{closed sub-F-scheme} of an F-scheme $X$ is a ringed space of the form
$(Z(\mathcal{I}),\mathcal{O}_X/\mathcal{I})$ for some quasi-coherent sheaf  $\mathcal{I}$ of ideals.
A morphism $f:X\to Y$ is called a \textbf{closed immersion} if it is   equivalent to a morphism
of type $(Z(\mathcal{I}),\mathcal{O}_Y/\mathcal{I})\to Y$ (the inclusion morphism) for some
quasi-coherent
sheaf of ideals $\mathcal{I}$ on $Y$. Recall that  $f:X\to Y$ and $f': X' \to Y$ are called
equivalent if there is an isomorphism $i: X' \to X$ such that $f'=f i$.
\end{definition}
Note that any closed sub-F-scheme of an affine F-scheme $\text{F}(A)$ is
of the form $\text{F}(A/I)$
for some ideal $I$ of $A$. Moreover, the morphism $ \text{F}(A/I)\to \text{F}(A)$
(induce by the quotient homomorphism) is a closed immersion.

A locally closed immersion (and sub-F-scheme) is what one expects;
a morphism $f:X\to Y$ is called a \textbf{locally closed immersion} if
there are a closed immersion $g:X\to Z$ and an open immersion $
h: Z\to Y$ such that $f=h g$.
A \textbf{locally closed sub-F-scheme} of an F-scheme $X$ is
a closed sub-F-scheme of an open sub-F-scheme of $X$.

To give another characterization of closed immersions we need the following lemma.
\begin{lemma}\label{kernelquasicoherent}
Suppose that  $(Y,\mathcal{O}_Y)=(Z(\mathcal{I}),\mathcal{O}_X/\mathcal{I})$ is a
closed sub-F-scheme of an
F-scheme $X$ and $i:Y\to X$ is the inclusion morphism. Then the homomorphism
$i_\sharp: \mathcal{O}_X\to i_*\mathcal{O}_Y$ is onto and its kernel is equal to
$\mathcal{I}$.
\end{lemma}
\begin{proof}
This is a local question which
immediately follows from the proof of Proposition \ref{closedsubfschemes P}.
\end{proof}
Using this lemma, we have the following characterization of closed immersions.
\begin{proposition}\label{charaterizationofclosedimmersions P}
A morphism $f:X\to Y$ of F-schemes is a closed immersion iff $f$ induces a homeomorphism from
$X$ onto a semi-closed subset of $Y$ and the homomorphism $f_\sharp:\mathcal{O}_Y\to f_*\mathcal{O}_X$
is onto.
\end{proposition}
\begin{proof}
One direction is straightforward. So suppose that $f$ satisfies the above conditions.
Let $\mathcal{I}$ be the kernel of $\mathcal{O}_Y\to f_*\mathcal{O}_X$ which is then
a sheaf of ideals on $Y$. I claim that  $\mathcal{I}$ is quasi-coherent and $f$ is canonically
equivalent to
the inclusion homomorphism $(Z(\mathcal{I}),\mathcal{O}_Y/\mathcal{I})\to Y$ which proves the
proposition.
To prove this, it is enough to assume that $Y=\text{F}(A)$ is affine.  Since each semi-closed
subset of $Y$ has a center, $X$ has to have a center which implies that $X$ is affine, see Lemma
\ref{beingaffine L}. Therefore $X=\text{F}(B)$ and $f:X\to Y$ is induced by a homomorphism
$\phi:A\to B$. But $\phi:A\to B$ is the homomorphism
 $f^\sharp_{\text{U}(A)}:\mathcal{O}_{A,\text{U}(A)}
 \to (f_*\mathcal{O}_B)_{\text{U}(A)}$ which is onto. By Lemma \ref{kernelquasicoherent},
 $\mathcal{I}$ is quasi-coherent and $f$ is canonically equivalent to  the inclusion homomorphism $(Z(\mathcal{I}),\mathcal{O}_Y/\mathcal{I})\to X$.
\end{proof}
The following corollary is immediate from the above proposition.
\begin{corollary}\label{closedimmersionopenconver C}
Suppose that $f:X\to Y$ is a morphism of F-schemes and $\{U_i\}$ is an open covering of $Y$. Then the
morphism $f$
is a closed immersion iff $f^{-1}(U_i)\to U_i$ is a closed immersion for each $i$.
\end{corollary}
The composition of two open immersions is clearly an open immersion. The same is
valid for closed immersions.
\begin{proposition}\label{compositionofclosedimmersions P}
The composition of
closed immersions  is
  a closed immersion.
\end{proposition}
\begin{proof}
Suppose that $f:X\to Y$ and $g:Y\to Z$ are closed immersions. To prove that $gf$ is a closed
immersion,
by Corollary \ref{closedimmersionopenconver C}, we may assume that $Z=\text{F}(A)$ is affine.
This implies that $X$ and $Y$, being closed sub-F-schemes of affine F-schemes,
are also affine.  So $Y=\text{F}(A/I)$ and $X=\text{F}(A/I')$ for some ideals $I\subset I'$.
Therefore   $g f: \text{F}(A/I')\to A$ is a closed immersion.
 \end{proof}

\end{subsection}

 %%%%%%%%%%%%%%%%%%%%%%%%%%%%%%%%%%%%%%%%%%%%%%%%%%%%%%%%%%%%%%%%%%%%%%%%%%%%%%%%%%%%%%%%%%%%%%%%
%%%%%%%%%%%%%%%%%%%%%%%%%%%%%%%%%%%%%%%%%%%%%%%%%%%%%%%%%%%%%%%%%%%%%%%%%%%%%%%%%%%%%%%%%%%%%%%%
%%%%%%%%%%%%%%%%%%%%%%%%%%%%%%%%%%%%%%%%%%%%%%%%%%%%%%%%%%%%%%%%%%%%%%%%%%%%%%%%%%%%%%%%%%%%%%%%
%%%%%%%%%%%%%%%%%%%%%%%%%%%%%%%%%%%%%%%%%%%%%%%%%%%%%%%%%%%%%%%%%%%%%%%%%%%%%%%%%%%%%%%%%%%%%%%%
%%%%%%%%%%%%%%%%%%%%%%%%%%%%%%%%%%%%%%%%%%%%%%%%%%%%%%%%%%%%%%%%%%%%%%%%%%%%%%%%%%%%%%%%%%%%%%%%

\begin{subsection}{Separated F-schemes}
Let \textbf{FSch} be the category of F-schemes.
Given any F-scheme $Y$, the category of F-schemes over $Y$ is
denoted by \textbf{FSch}$_Y$. It is easy to see that $\text{F}(\mathbb{Z})$
is the final object in \textbf{FSch} and therefore
any F-scheme is naturally an F-scheme over $\mathbb{Z}$.
The following theorem states that  fiber products exist  in \textbf{FSch}$_Y$.
\begin{theorem}\label{fiberproduct}
For any two F-schemes $X$ and $Z$   over the F-scheme $Y$,
the fiber product $X\times_Y Z$ exists and is unique up to isomorphism.
\end{theorem}
\begin{proof}
The proof is quite similar to the proof of the existence of fiber products
in the category of schemes, see \cite{Ha}.
I only mention that
in the case $X=\text{F}(A)$, $Z=\text{F}(B)$ and $Y=\text{F}(R)$  are  all affine,
$ X\times_Y Z\cong \text{F}(A*_R B)$    where $A*_R B $ is the coproduct of $A$ and $B$
over $R$.
\end{proof}
In the particular case $X=Z$, from the universal property of fiber product,
 we obtain a morphism
$\Delta_X:X\to X\times_Y X$, called the \textbf{diagonal morphism},
such that $p_1\Delta_X=p_2\Delta_X=1_X$ where
$p_i:X\times_Y X\to X$ are the projection morphisms.

Having proved the existence of fiber products, one can extend
the important notion of base extension in scheme theory (see \cite{Ha})
to the case of F-schemes.   The following proposition says that closed immersions are
stable under base extension.
\begin{proposition}\label{closedimemrsionbaseextension P}
Suppose that $f:Y'\to Y$ is a closed immersion. For any morphism $g:X\to  Y$, the natural
morphism
$X\times_{Y} Y'\to X$ is a closed immersion.
\end{proposition}
\begin{proof}
Using Corollary \ref{closedimmersionopenconver C}, it is enough to prove this for the case
$X=\text{F}(A)$ is affine. Then there is an open affine subset $U=\text{F}(B)$ of $Y$
such that $g(X)\subset U$ (since $X$ has a center). It can be seen that
the natural morphism $X\times_U f^{-1}(U)\to X\times_Y Y'$ is an isomorphism. So we can,
furthermore,
assume that $Y=\text{F}(B)$ is affine. So $Y'=\text{F}(B/I)$ for some ideal $I$ of $B$.
Then  $X\times_{Y} Y'\to X$ is nothing
but the morphism induced by the natural  homomorphism $A*_B (B/I)\to A$ which is onto.
So   $X\times_{Y} Y'\to X$ is a closed immersion.
\end{proof}
Similarly one can prove the following proposition.
\begin{proposition}\label{closedimemrsionproduct P}
Suppose that $X'\to X$ and $Y'\to Y$   are  closed immersions in the
category \textbf{FSch}$_Z$.
The natural morphism $X'\times_Z Y'\to X\times_Z Y$
is a closed immersion.
\end{proposition}
Now we are in a position to introduce the notion of separated F-schemes.
\begin{definition}
Suppose that $f:X\to Y$ is a morphism of  F-schemes. Then $f$ is called \textbf{separated}
if the diagonal morphism $\Delta_X:X\to X\times_Y X$ is a closed immersion. In this case,
$X$ is called \textbf{separated
over $Y$}.
The F-scheme $X$ is called \textbf{separated} if it is separated over $\text{F}(\mathbb{Z})$.
\end{definition}
The following lemma is easy to prove.
\begin{lemma}\label{affineseparated L}
A morphism of affine F-schemes is separated.

\end{lemma}
\begin{proof}
Let $f:\text{F}(B)\to\text{F}(A)$ be the morphism induced by
the homomorphism $\phi:A\to B$. Then  the diagonal morphism
$$\Delta:\text{F}(B)\to \text{F}(B)\times_{\text{F}(A)} \text{F}(B)=\text{F}(A*_B A)$$
 is
the morphism induced by
 the natural homomorphism $A*_B A\to A$ which is clearly
 onto.
\end{proof}
The following proposition gives a simple characterization of separated morphisms.
\begin{proposition}\label{sepatated P}
A morphism $f:X\to Y$ is separated iff the image of the diagonal morphism $\Delta_X$ is
semi-closed in $X\times_Y X$.
\end{proposition}
\begin{proof}
One direction is trivial. The other direction follows from Proposition
\ref{charaterizationofclosedimmersions P} and the affine case.
\end{proof}
Similar to the case of schemes, one has the following result.
\begin{corollary}\label{intersectionofaffinesinsepeareted C}
In a separated F-scheme, the intersection
of any two open affine sub-F-schemes is affine. In
particular,  every separated F-scheme is either affine or 1-affine.
\end{corollary}
Some of the properties of separated morphisms are given in the following proposition.
\begin{proposition}\label{separatedmporphismsbaseextension P}
\begin{enumerate}
\item Open and closed immersions are separated.
\item  Separated morphisms are stable under base extension.
\item  If $X'\to X$ and $Y'\to Y$   are  separated morphisms in the
category \textbf{FSch}$_Z$. Then
the natural morphism $X'\times_Z Y'\to X\times_Z Y$
is also separated.
\item The composition of two separated morphisms is separated.
\item If $f:X\to Y$ and $g:Y\to Z$ are two morphisms such that
$gf$ is separated then  $f$ is also separated.
\item A morphism $f:X\to Y$ is a separated morphism iff there is an open covering $U_i$ of $Y$ such that
each induced morphism $f^{-1}(U_i)\to U_i$ is separated.
\end{enumerate}
\end{proposition}
\begin{proof}
Having proved the corresponding facts about closed immersions,
one can see that the proof of this proposition is essentially the same as the one
for schemes, see \cite{Gr}.
\end{proof}
\begin{remark}
It is known that if $\mathcal{P}$ is a property of
of morphisms of schemes which satisfies the following
\begin{enumerate}
\item a closed immersion has $\mathcal{P}$,
\item a composition of two morphisms having $\mathcal{P}$ has $\mathcal{P}$,
\item $\mathcal{P}$ is stable under base extension,
\end{enumerate}
then $\mathcal{P}$ satisfies the following as well
\begin{enumerate}
\item a product of morphisms having $\mathcal{P}$ has $\mathcal{P}$,
\item if $f:X\to Y$ and $g:Y\to Z$ are two morphisms, and if $gf$ has $\mathcal{P}$ and $g$ is
separated then  $f$ has $\mathcal{P}$.
\end{enumerate}
It can be seen that the same holds for F-schemes.
\end{remark}
In the following example, I give an example of a   nonseparated F-scheme.
\begin{example}
Suppose that $A$ is a ring for which $\text{F}(A)\setminus\{\text{U}(A)\}$ is
not an affine F-scheme, e.g. $A=k\langle x, y\rangle$ for any commutative field $k$.
Let $X$ be the F-scheme obtained by gluing $ \text{F}(A) $ with itself
along $\text{F}(A)\setminus\{\text{U}(A)\}$, see Example \ref{nonseparatdexample E}.
The F-scheme $X$ is not separated because there are open affine
subsets of $X$ (i.e. copies of $ \text{F}(A) $)
 whose intersection is not affine, see Corollary
 \ref{intersectionofaffinesinsepeareted C}.

\end{example}

\end{subsection}

%%%%%%%%%%%%%%%%%%%%%%%%%%%%%%%%%%%%%%%%%%%%%%%%%%%%%%%%%%%%%%%%%%%%%%%%%%%%%%%%%%%%%%%%%%%%%%%%
%%%%%%%%%%%%%%%%%%%%%%%%%%%%%%%%%%%%%%%%%%%%%%%%%%%%%%%%%%%%%%%%%%%%%%%%%%%%%%%%%%%%%%%%%%%%%%%%
%%%%%%%%%%%%%%%%%%%%%%%%%%%%%%%%%%%%%%%%%%%%%%%%%%%%%%%%%%%%%%%%%%%%%%%%%%%%%%%%%%%%%%%%%%%%%%%%
%%%%%%%%%%%%%%%%%%%%%%%%%%%%%%%%%%%%%%%%%%%%%%%%%%%%%%%%%%%%%%%%%%%%%%%%%%%%%%%%%%%%%%%%%%%%%%%%
%%%%%%%%%%%%%%%%%%%%%%%%%%%%%%%%%%%%%%%%%%%%%%%%%%%%%%%%%%%%%%%%%%%%%%%%%%%%%%%%%%%%%%%%%%%%%%%%

\begin{subsection}{More properties of F-schemes}
There are other properties and concepts in the theory of schemes
that one can try to extend to F-schemes. In this part some of
these properties are considered. Before discussing these concepts, I want
to explain an important property of  many of these concepts.
It is well known that for a class of properties of schemes
(or morphisms of schemes), if an open affine covering of a scheme has those properties
then any open affine covering of that scheme has the same properties.
Some of these properties are  ``locally noetherian schemes, finite morphisms,
morphisms of finite type,...'', see \cite{Ha}.
Most of the concepts considered in this part have the same property. However, it is
a lot easier to prove this property for these concepts in the case of F-schemes
(because the topology of an F-scheme is more restricted than the topology of a scheme).

\begin{definition}
Suppose that $f:X\to Y$
is a morphism of F-schemes.
\begin{enumerate}
\item The morphism $f$ is called \textbf{affine} if there exists an open affine covering $U_i$
of $Y$ such that $f^{-1}(U_i)$ is affine for each $i$.
 \item The morphism $f$ is called \textbf{quasi-compact} if there exists an open affine covering $U_i$
of $Y$ such that $f^{-1}(U_i)$ is quasi-compact for each $i$.
\item The morphism $f$ is called \textbf{locally of finite type}   if there
exists an open affine covering $U_i=\text{F}(A_i)$
of $Y$ such that for
each $i$, the open subset $f^{-1}(U_i)$ can be covered by open affine subsets
$V_{ij} = \text{F}(B_{ij})$   where
each $B_{ij}$ is  finitely generated over $A_i$ as a ring.  The morphism $f$ is of \textbf{finite
type} if in addition each $f^{-1}(U_i)$ can be covered by a finite number of the
 $V_{ij}$.
%\item  The morphism $f$ is called a \textbf{left-finite morphism} if there
%exists an open affine covering $U_i=\text{F}(A_i)$
%of $Y$ such that for
%each $i$, the open subset $f^{-1}(U_i)=\text{F}(B_i)$ is affine and
% each $B_{i}$ is a finitely generated left $A_i$-module.

\end{enumerate}
\end{definition}
Before discussing these concepts, we need a few lemmas.
\begin{lemma}\label{quasi-compact L}
An F-scheme is quasi-compact iff it can be covered by finitely many
open affine subsets.
\end{lemma}
\begin{proof}
One direction is trivial. So suppose that $X$ is an F-scheme such that $X=\cup_{i=1}^{n}X_i$
for some open affine subsets $X_i$.  Let $x_i$ be the center of $X_i$. If $U_j$
is an open covering of $X$ then there are $U_{j_1},...,U_{j_n}$ such that $x_i\in U_{j_i}$
for each $i$. So $X_i\subset U_{j_i}$
for each $i$ and hence $X=\cup_{i=1}^{n}U_{j_i}$.
\end{proof}
\begin{lemma}\label{morphismsofaffineaffine L}
A morphism of affine F-schemes is affine.
\end{lemma}
\begin{proof}
This follows from the facts that any open affine subset of an F-scheme
is a fundamental open subset (see Lemma \ref{beingaffine L})
and the inverse image of a fundamental open
subset is a fundamental open subset under a morphism of affine F-schemes.
\end{proof}
Note that any affine morphism is quasi-compact.
Now, regarding the above  properties, one has the following propositions.
\begin{proposition}\label{anyopencoveringaffine P}
A morphism  $f:X\to Y$ of F-schemes is affine iff for \underline{any} open affine subset
$U$ of $Y$, the open subset $f^{-1}(U)$ is affine.
\end{proposition}
\begin{proof}
One direction is trivial. So suppose that $f$ is affine  and  $U=\text{F}(A)$
is an open affine subset of $Y$. Since $U$ has a center and $f$ is affine, there must be an open
affine subset $U\subset V=\text{F}(B)$, such that $f^{-1}(V)$ is affine. This implies that
$U$ is a fundamental open subset of $V$, and hence $f^{-1}(U)$ is a fundamental open subset of
$f^{-1}(V)$. Therefore  $f^{-1}(U)$ is affine.
\end{proof}
\begin{proposition}\label{anyopencoveringquasi-compact P}
A morphism  $f:X\to Y$ of F-schemes is quasi-compact iff for \underline{any} open affine subset
$U$ of $Y$, the open subset $f^{-1}(U)$ is quasi-compact.
\end{proposition}
\begin{proof}
One direction is trivial. So suppose that $f$ is affine  and  $U=\text{F}(A)$
is an open affine subset of $Y$. Since $U$ has a center and $f$ is affine, there must be an open
affine subset $U\subset V=\text{F}(B)$ such that $f^{-1}(V)$ is quasi-compact. Hence
$f^{-1}(V)$ can be covered by finitely many open affine subsets. Since $U$ is
a fundamental open subset of $V$, the subspace $f^{-1}(U)$ is a union of finitely many
open affine subsets and hence is quasi-compact by Lemma \ref{quasi-compact L}.
\end{proof}
\begin{proposition}\label{anyopencoveringlocallyfinitetype P}
A morphism  $f:X\to Y$ of F-schemes is locally of finite type
iff for  any open affine subsets
$U=\text{F}(A)$ of $Y$ and $V=\text{F}(B)$ of $X$
such that $f(V)\subset U$,  the ring $B$ is  finitely generated over $A$ as a ring.
\end{proposition}

\begin{proof}
One direction is trivial. So suppose that $U=\text{F}(A)\subset Y$ and $V=\text{F}(B)\subset X$
are open affine subsets such that $f(U)\subset V$. Then there are open
affine subsets $U\subset U'=\text{F}(A')\subset Y$ and $V\subset V' =\text{F}(B')\subset X$
such that  $B'$ is finitely generated over $A'$ as a ring (because $f$ is locally
of finite type and $U$ and $V$ have centers). Now $A=A'_{\begin{bf} a \end{bf}}$
for some fraction $\begin{bf} a \end{bf}$ of $A'$. Therefore
$B=B'_{\phi(\begin{bf} a \end{bf})}$ where $\phi:A'\to B'$ is the homomorphism
which induces the morphism $f:V'\to U'$. But it is easy to see that
$B=B'_{\phi(\begin{bf} a \end{bf})}$ is  finitely generated over
$A=A'_{\begin{bf} a \end{bf}}$ as a ring because
 $B'$ is finitely generated over $A'$ as a ring.
\end{proof}

\begin{proposition}\label{anyopencoveringfinitetype P}
A morphism  $f:X\to Y$ of F-schemes is  of finite type
iff  it is quasi-compact and locally of finite type.
\end{proposition}
\begin{proof}
It follows from Propositions \ref{anyopencoveringlocallyfinitetype P}
and \ref{anyopencoveringquasi-compact P}.

\end{proof}

%\begin{proposition}\label{anyopencoveringfinite  P}
%A morphism  $f:X\to Y$ of F-schemes is  a left-finite morphisms
%iff for  any open affine subset
%$U=\text{F}(A)$ of $Y$, we have   $f^{-1}(U)=\text{F}(B)$
%is affine and   the ring $B$ is  a finitely generated  left $A$-module.
%\end{proposition}
%\begin{proof}
%One direction is trivial. So suppose that $f$ is a left-finite morphism and
%$U=\text{F}(A)$ is an open affine subset of $Y$. Then   there is an open
%affine subset $U\subset U'=\text{F}(A')\subset Y$ and $f^{-1}(U') =\text{F}(B')\subset X$
%is affine where  $B'$ is a finitely generated  left $A'$-module. Now $A=A'_{\begin{bf} a \end{bf}}$
%for some fraction $\begin{bf} a \end{bf}$ of $A'$. Therefore
%$f^{-1}(U)=\text{F}(B'_{\phi(\begin{bf} a \end{bf})})$ is affine  where $\phi:A'\to B'$ is the homomorphism
%which induces the morphism $f:f^{-1}(U')\to U'$.
%
%
%\end{proof}

The following proposition is straightforward to prove.
\begin{proposition}\label{closedimmersionsproperties P}
A closed immersion is affine, quasi-compact and of finite type.
\end{proposition}

\end{subsection}

%%%%%%%%%%%%%%%%%%%%%%%%%%%%%%%%%%%%%%%%%%%%%%%%%%%%%%%%%%%%%%%%%%%%%%%%%%%%%%%%%%%%%%%%%%%%%%%%
%%%%%%%%%%%%%%%%%%%%%%%%%%%%%%%%%%%%%%%%%%%%%%%%%%%%%%%%%%%%%%%%%%%%%%%%%%%%%%%%%%%%%%%%%%%%%%%%
%%%%%%%%%%%%%%%%%%%%%%%%%%%%%%%%%%%%%%%%%%%%%%%%%%%%%%%%%%%%%%%%%%%%%%%%%%%%%%%%%%%%%%%%%%%%%%%%
%%%%%%%%%%%%%%%%%%%%%%%%%%%%%%%%%%%%%%%%%%%%%%%%%%%%%%%%%%%%%%%%%%%%%%%%%%%%%%%%%%%%%%%%%%%%%%%%

\end{section}

%%%%%%%%%%%%%%%%%%%%%%%%%%%%%%%%%%%%%%%%%%%%%%%%%%%%%%%%%%%%%%%%%%%%%%%%%%%%%%%%%%%%%%%%%%%%%%%%
%%%%%%%%%%%%%%%%%%%%%%%%%%%%%%%%%%%%%%%%%%%%%%%%%%%%%%%%%%%%%%%%%%%%%%%%%%%%%%%%%%%%%%%%%%%%%%%%
%%%%%%%%%%%%%%%%%%%%%%%%%%%%%%%%%%%%%%%%%%%%%%%%%%%%%%%%%%%%%%%%%%%%%%%%%%%%%%%%%%%%%%%%%%%%%%%%
%%%%%%%%%%%%%%%%%%%%%%%%%%%%%%%%%%%%%%%%%%%%%%%%%%%%%%%%%%%%%%%%%%%%%%%%%%%%%%%%%%%%%%%%%%%%%%%%
%%%%%%%%%%%%%%%%%%%%%%%%%%%%%%%%%%%%%%%%%%%%%%%%%%%%%%%%%%%%%%%%%%%%%%%%%%%%%%%%%%%%%%%%%%%%%%%%
%%%%%%%%%%%%%%%%%%%%%%%%%%%%%%%%%%%%%%%%%%%%%%%%%%%%%%%%%%%%%%%%%%%%%%%%%%%%%%%%%%%%%%%%%%%%%%%%
%%%%%%%%%%%%%%%%%%%%%%%%%%%%%%%%%%%%%%%%%%%%%%%%%%%%%%%%%%%%%%%%%%%%%%%%%%%%%%%%%%%%%%%%%%%%%%%%
%%%%%%%%%%%%%%%%%%%%%%%%%%%%%%%%%%%%%%%%%%%%%%%%%%%%%%%%%%%%%%%%%%%%%%%%%%%%%%%%%%%%%%%%%%%%%%%%
%%%%%%%%%%%%%%%%%%%%%%%%%%%%%%%%%%%%%%%%%%%%%%%%%%%%%%%%%%%%%%%%%%%%%%%%%%%%%%%%%%%%%%%%%%%%%%%%
%%%%%%%%%%%%%%%%%%%%%%%%%%%%%%%%%%%%%%%%%%%%%%%%%%%%%%%%%%%%%%%%%%%%%%%%%%%%%%%%%%%%%%%%%%%%%%%%

\begin{section}{Projective F-schemes}
Projective schemes are fundamental objects in algebraic
geometry, see \cite{Ha}. In this part their counterparts in the theory of F-schemes
, namely projective F-schemes, are defined. For a detailed discussion of graded rings see
\cite{NV} (and references therein).

%%%%%%%%%%%%%%%%%%%%%%%%%%%%%%%%%%%%%%%%%%%%%%%%%%%%%%%%%%%%%%%%%%%%%%%%%%%%%%%%%%%%%%%%%%%%%%%%
%%%%%%%%%%%%%%%%%%%%%%%%%%%%%%%%%%%%%%%%%%%%%%%%%%%%%%%%%%%%%%%%%%%%%%%%%%%%%%%%%%%%%%%%%%%%%%%%

\begin{subsection}{Graded rings and localization}
Suppose that $R=\bigoplus_{n\in\mathbb{Z}} R_n$ is a $\mathbb{Z}$-graded
(possibly noncommutative) ring.
Recall that an element $r\in R$ is called homogeneous of degree $n$ if $r\in R_n$.
Every element $r\in R$ can be uniquely written as $r=r_{i_1}+\cdots+r_{i_n}$ where
$i_1<i_2<...<i_n$ and each $r_{i_j}$ is homogeneous of degree $i_j$. The homogeneous elements
$r_{i_j}$'s are called the homogeneous components of $r$. A subset of $R$ is called
homogeneous if it only contains homogeneous elements of $R$. Given a subset $S$ of $R$, set $h(S)$
to be the set of homogeneous elements in $S$.
A graded homomorphism of graded rings is a ring homomorphism $\phi:R\to T$
of graded rings $R=\bigoplus_{n\in\mathbb{Z}} R_n$ and
$T=\bigoplus_{n\in\mathbb{Z}} T_n$ with the property $\phi(R_n)\subset T_n$ for
all $n\in \mathbb{Z}$.

There is a ``homogeneously'' localizing process for graded rings similar
to the one for rings, see Section \ref{localization S}. Here a short account of
this processes (with similar proofs omitted) is given.
Suppose that $R$ is a graded ring and $S\subset R$ is
 a homogeneous subset of $R$. Then the localization
ring $R_{S}$ is naturally a $\mathbb{Z}$-graded ring. To see this, recall that
$R_{S}$ has the following presentation
$$R_{S}=\langle R,\{x_s\}_{s\in S}|\, x_ss=sx_s=1,\,\text{for any}\, s\in S \rangle. $$
Setting $\deg(x_s)=-\deg(s)$ for $s\in S$, we obtain a $\mathbb{Z}$-graded structure on
$R_{S}$ compatible with the $\mathbb{Z}$-graded structure of $R$, i.e.
the natural homomorphism $i_S:R\to R_{S}$ is a graded homomorphism. Moreover the natural
homomorphism $i_S:R\to R_{S}$ has the following universal property in the category of
graded rings (morphisms being graded homomorphisms): for any graded ring $T$ and a graded
homomorphism $\phi:R\to T$ such that $\phi(s)$ is invertible in $T$ for all $s\in S$, there
is a unique \underline{graded} homomorphism $\phi_S:R_S\to T$ such that $\phi=\phi_Si_S$.
 \begin{definition}
Suppose that $R$ is a  graded ring. A homogeneous subset $S$ of $R$ is called
a \textbf{homogeneously properly invertible subset} if
there is a  graded homomorphism $\phi:R\to T$, from $R$  to a \underline{nonzero}  graded ring $T$,
such that for every $s\in S$, the element $\phi(s)$ is invertible in $T$.
\end{definition}
\begin{definition}
A  homogeneous subset $S$ of  $R$ is called a
\textbf{homogeneously fully invertible subset} if there is a graded
homomorphism $\phi:R\to T$, from $R$ to a \underline{nonzero} graded ring $T$, such that
$S=h(\phi^{-1}(\text{U}(T)))$.
\end{definition}
\begin{lemma}\label{gradedfully invertiblesubsets L}
A  homogeneous subset $S$ of $R$ is homogeneously fully invertible iff
$S$ is homogeneously properly invertible and $S=h(R\cap \text{U}(R_S))$.
\end{lemma}
\begin{proposition}\label{gradedirrinv P}
Suppose that $S$ and $T$ are two homogeneously fully invertible  subsets of $R$. Then
\begin{enumerate}
\item    $h(\text{U}(R))\subset S$.
\item $S$ is multiplicatively closed, i.e. if $s,t\in S$ then
 $st\in S$.
\item If $aba\in S$ for some homogeneous elements $a,b\in R$ then  $a,b\in S$. In particular,  $R_{aba}=R_{\{a,b\}}$.
\item $S\subset T$ iff there is a graded homomorphism $R_S\to R_T$ making the following
diagram commutative
\begin{displaymath}
\xymatrix{ & R \ar[dr]_{i_T}  \ar[r]^{i_S} & R_S \ar[d] \\
           &   & R_T}
 \end{displaymath}
\end{enumerate}
\end{proposition}
\begin{proposition}\label{gradedpullback P}
Suppose that $\phi:R\to T$ is a graded homomorphism and $S$ is a
homogeneously fully invertible  subset of $T$.
Then $h(R\cap S)$ is a homogeneously fully invertible subset of $R$.
Moreover there is a unique homomorphism  $\phi^S:R_{h(R\cap S)}\to T_S$ making the following
diagram commutative
\begin{displaymath}
\xymatrix{ & R \ar[d]_{i_{h(R\cap S)}}  \ar[r]^{\phi} & T \ar[d]^{i_S} \\
           & R_{h(R\cap S)}\ar[r]^{\phi^S}  & T_S}
           \end{displaymath}
\end{proposition}
Since for a homogeneously fully invertible subset $S$, the localization $R_S$ is graded, it is
straightforward to define the notion of \textbf{homogeneously fully invertible systems}.
Given a graded ring homomorphism $\phi:R\to T$, one obtains a sequence $\mathcal{S}=(S_1,S_2,...)$
such that  each $S_i$ is
a homogeneously fully invertible subset of the graded   ring
$R_{\mathcal{S},n}$ where $R_{\mathcal{S},n}=(R_{\mathcal{S},n-1})_{S_n}$
and $R_{\mathcal{S},0}=R$. Given a homogeneously fully invertible system $\mathcal{S}$,
the ring
$R_{\mathcal{S}}=\underset{n}{\underrightarrow{\lim}} R_{\mathcal{S},n}$ has a graded structure
coming from those of $R_{\mathcal{S},n}$'s.
\begin{lemma}
Let $\mathcal{S}=(S_1,S_2,...)$ be a homogeneously fully invertible system on $R$.
Then for each $n$, $\mathcal{S}_n=(S_{n+1},S_{n+2},...)$
is a homogeneously fully invertible system
on $R_{\mathcal{S},n}$.
\end{lemma}
\begin{proposition}
Let $R$ be a graded ring and $\mathcal{S}=(S_1,S_2,...)$ be a sequence of homogeneously fully
invertible subsets $S_n$ of rings
$R_{\mathcal{S},n}=(R_{\mathcal{S},n-1})_{S_{n}}$ where $n\geq 1$ and $R_{\mathcal{S},0}=R$.
Then $\mathcal{S} $ is a homogeneously fully invertible system on $R$
iff  for any $m> n\geq 0$, one has    $h(R_{\mathcal{S},n}\cap S_m)=S_{n+1}$.
\end{proposition}
\begin{proposition}\label{gradedpullbacksystem P}
Suppose that $\phi:R\to T$ is a graded ring homomorphism and $\mathcal{S}=(S_1,S_2,...)$
is a homogeneously fully invertible system on $T$.
Then  the sequence
$$h(\phi^{*}(\mathcal{S})):=(h(R\cap S_1),h(R_{h(R\cap S_1)}\cap S_2),...)$$
is a homogeneously fully invertible system on $R$, called
the pull back of $\mathcal{S}$ to $R$. Moreover
\begin{enumerate}
\item There  is  a unique graded ring homomorphism
$\phi^\mathcal{S}:R_{h(\phi^*(\mathcal{S}))}
\to T_\mathcal{S}$
 making the following diagram commutative
\begin{displaymath}
\xymatrix{ & R \ar[d]_{i_{h(\phi^*(\mathcal{S}))}}  \ar[r]^{\phi} & T \ar[d]^{i_\mathcal{S}} \\
           & R_{h(\phi^*(\mathcal{S}))}\ar[r]^{\phi^\mathcal{S}}  & T_\mathcal{S}}
 \end{displaymath}
\item The homomorphism  $\phi^\mathcal{S}:R_{h(\phi^*(\mathcal{S}))}\to T_\mathcal{S}$
is a \textbf{homogeneously local} homomorphism, i.e. the image of a homogeneous element
of $R_{h(\phi^*(\mathcal{S}))}$ is invertible in $T_\mathcal{S}$
iff it is invertible in $R_{h(\phi^*(\mathcal{S}))}$.
\end{enumerate}
\end{proposition}
\begin{proposition}\label{F.I.SinA_Sgraded P}
For any  homogeneously fully invertible system $\mathcal{S}$ on a graded ring $R$, there is a one-to-one correspondence
between homogeneously fully invertible systems on $R_\mathcal{S}$ and
homogeneously fully invertible systems on $R$
containing $\mathcal{S}$, given by the pull back
map
$\mathcal{T}\mapsto h(i_\mathcal{S}^{*}(\mathcal{T}))$. Moreover, for each homogeneously  fully invertible
system $\mathcal{T}$ on
 $R_\mathcal{S}$, the natural graded ring homomorphism $R_{h(i_\mathcal{S}^{*}(\mathcal{T}))}\to
 (R_\mathcal{S})_{\mathcal{T}}$ is
 an isomorphism.
\end{proposition}
A fraction $\begin{bf} a \end{bf} =(a_1,...,a_n)$ of $R$ is called
\textbf{homogeneous} if each $a_i$ is
a homogeneous element of $R_{\begin{bf} a \end{bf},i-1}$ (note that
the localization rings $R_{\begin{bf} a \end{bf},i-1}$ are naturally graded).
Given a homogeneously fully invertible
system $\mathcal{S}$ on $R$, set $R_{(\mathcal{S})}$
 to be the set of  homogeneous elements of  degree zero  in
$R_{\mathcal{S}}$. Clearly $R_{(\mathcal{S})}$ is  a subring of $R_{\mathcal{S}}$.
Likewise, given a homogeneous fraction
$\begin{bf} a \end{bf} $ of $R$, set $R_{(\begin{bf} a \end{bf})}$ to be the subring of
homogeneous  elements of  degree zero in $R_{\begin{bf} a \end{bf}}$.

Given any ring $A$, there is a trivial graded structure on $A$
in which all (nonzero) elements of $A$ are homogeneous of degree zero.
With this graded structure, the notions of fully invertible systems and
homogeneously fully invertible systems on $A$ coincide.

A graded ring $R=\bigoplus_{n\in\mathbb{Z}} R_n$ is called  a
$\mathbb{Z}$-crossed product ring if
$R$ has an invertible homogeneous element of degree one (see \cite{NV}).
To construct an example of a $\mathbb{Z}$-crossed product ring,
let's consider a ring $A$ with an automorphism
$\sigma:A\to A$. The twisted polynomial ring $A[x,x^{-1};\sigma]$
is defined to be the ring with the following presentation
$$A[x,x^{-1};\sigma]=\langle A,x,x^{-1}| xx^{-1}=x^{-1}x=1,\,
xa=\sigma(a)x \,\, \text{for all}\,\, a\in A \rangle. $$
So elements of $A[x,x^{-1};\sigma]$ are left polynomials of the form
$$a_m x^{m}+\cdots+a_{-1}x^{-1}+a_0+a_1x+\cdots+a_n x^n$$
for some $m\leq 0\leq n$ and $a_i\in A$. The twisted  polynomial ring $A[x,x^{-1};\sigma]$
has an obvious $\mathbb{Z}$-graded structure in which the set of homogeneous
elements of degree $n$ is the set $Ax^{n}=\{ax^n|a\in A\}$. Since $x$ is
an invertible homogeneous element of degree one, the ring $A[x,x^{-1};\sigma]$
is a $\mathbb{Z}$-crossed product ring. The following lemma shows that every
$\mathbb{Z}$-crossed product
ring is a
twisted polynomial ring as above, see also \cite{NV}.
\begin{lemma}
For any $\mathbb{Z}$-crossed product ring $R=\bigoplus_{n\in\mathbb{Z}} R_n$, there is an
isomorphism
$R\cong R_0[x,x^{-1};\sigma]$ of  graded rings, for some automorphism
$\sigma:R_0\to R_0$.
\end{lemma}
\begin{proof}
Let $r\in R_1$ be an invertible element. Then the map $\sigma:R_0\to R_0$
defined by $\sigma(a)=rar^{-1}$ is an automorphism of $R_0$. Moreover, it is easy to see that
$R\cong R_0[x,x^{-1};\sigma]$ where $r$ is mapped to $x$.
\end{proof}

\begin{proposition}\label{homogeneousof A_a P}
Let $R=\bigoplus_{n\in\mathbb{Z}} R_n$ be a $\mathbb{Z}$-crossed product ring.
Then there is a
one-to-one correspondence
between homogeneously fully invertible subsets of $R$ and
(homogeneously) fully invertible subsets of $R_{0}$ (with the trivial graded structure),
 given by the pull back map
$S\mapsto
R_{0}\cap S$. Moreover the natural ring homomorphism
$(R_{0})_{ R_{0}\cap S}\to R _{(S)}$
is an isomorphism for any homogeneously fully invertible subset $S$ on $R $.
\end{proposition}
\begin{proof}
Let $a\in R_1$ be an invertible element.
Suppose that $S$ and $S_1$ are two homogeneously fully invertible subsets of $R$ such that
$R_0\cap S=R_0 \cap S_1 $. If $s\in S$ then  $sa^{-\deg(s)}\in R_0\cap S$. This implies that
$sa^{-\deg(s)}\in S_1$. Since $a\in S_1$ and $S_1$ is multiplicatively closed,
we must have $s\in S_1$. This implies that $S\subset S_1$ and similarly we have $S_1\subset S$.
So $S=S_1$ and hence the pull back map
$S\mapsto
R_{0} \cap S$ is  one-to-one. Now suppose that $T$ is  a fully invertible subset
of $R_0$. Considering $T$ as a homogeneous subset of $R$, we can construct the graded ring
$R_T$. Let $S$ be the homogeneously fully invertible subset induced by
the natural homomorphism $R\to R_T$. I claim that $S=R_{0} \cap T$ (which proves that
$S\mapsto
R_{0} \cap S$ is onto).  In order to show this, I first prove the last assertion
of this proposition, namely the natural homomorphism $(R_0)_T\to R_{(S)}$ is an isomorphism.
Note that  $R_{(S)}$ is the subring of degree zero  homogeneous elements of $R_T$.
Suppose that $y=r_1x_{s_1}r_2x_{s_2}...r_{k}x_{s_k}r_{{k+1}}$ is
a degree zero homogeneous element of $R_{T}$, where $s_i\in S$ and each $r_i\in R$ is homogeneous
such that $\deg(r_1)+\cdots+\deg(r_{k+1})=0$. Then $y$ can be written as
$$y=(r_1a^{-\deg(r_1)})
(a^{\deg(r_1)}x_{s_1}a^{-\deg(r_1)})
(a^{\deg(r_1)}r_2a^{-\deg(r_1)-\deg(r_2)})$$
$$...(a^{\deg(r_1)+\cdots+\deg(r_k)}x_{s_k}a^{-\deg(r_1)-\cdots-\deg(r_k)})
(a^{-\deg(r_{k+1})}r_{{k+1}})$$
which is in the image of the
ring homomorphism $(R_0)_T\to R_{(S)}$,
because $a^nx_s a^{-n}=x_{a^nsa^{-n}}$ for any $s\in S$ and $n\in\mathbb{Z}$. This proves that
$(R_0)_T\to R_{(S)}$ is onto.
Using the presentations of $(R_0)_T$ and $R_S$, and a
similar argument, one can show that $(R_0)_T\to R_{(S)}$
is injective. So  $(R_0)_T\to R_{(S)}$ is an isomorphism.
To show that $T=R_{0}\cap S$, note that
an element $r\in R_0$ is invertible in $R_S$ iff it is invertible in $R_{(S)}$.
So, using the isomorphism $(R_0)_T\to R_{(S)}$, we have $r\in R_0$ is invertible in $R_S$
iff $r$ is invertible in $(R_0)_T$ iff $r\in S$. This proves that $T= R_0 \cap S$.
\end{proof}
Since for any  homogeneously fully invertible subset $S$ of $R$,
the natural homomorphism $(R_0)_{R_0\cap S}\to R_{(S)}$ is an isomorphism, the natural homomorphism
$(R_0)_{R_0\cap S}\to R_{S}$ is a local homomorphism. Using this fact and a proof by induction,
we obtain the following proposition as well.
\begin{proposition}\label{homogeneoussystemsof A_a P}
Suppose that $R=\bigoplus_n R_n$ is a $\mathbb{Z}$-crossed product ring.
Then there is a
one-to-one correspondence
between homogeneously fully invertible systems on $R$ and
(homogeneously) fully invertible systems on $R_{0}$ (with the trivial graded structure),
 given by the pull back map
$\mathcal{S}\mapsto h(i^*(\mathcal{S}))=i^*(\mathcal{S}) $
where $i:R_0\to R$ is the inclusion homomorphism. Moreover the natural ring homomorphism
$(R_{0})_{i^*(\mathcal{S})}\to R _{(\mathcal{S})}$
is an isomorphism for any homogeneously fully invertible system $\mathcal{S}$ on $R $.
\end{proposition}

\end{subsection}

%%%%%%%%%%%%%%%%%%%%%%%%%%%%%%%%%%%%%%%%%%%%%%%%%%%%%%%%%%%%%%%%%%%%%%%%%%%%%%%%%%%%%%%%%%%%%%%%
%%%%%%%%%%%%%%%%%%%%%%%%%%%%%%%%%%%%%%%%%%%%%%%%%%%%%%%%%%%%%%%%%%%%%%%%%%%%%%%%%%%%%%%%%%%%%%%%
%%%%%%%%%%%%%%%%%%%%%%%%%%%%%%%%%%%%%%%%%%%%%%%%%%%%%%%%%%%%%%%%%%%%%%%%%%%%%%%%%%%%%%%%%%%%%%%%
%%%%%%%%%%%%%%%%%%%%%%%%%%%%%%%%%%%%%%%%%%%%%%%%%%%%%%%%%%%%%%%%%%%%%%%%%%%%%%%%%%%%%%%%%%%%%%%%
%%%%%%%%%%%%%%%%%%%%%%%%%%%%%%%%%%%%%%%%%%%%%%%%%%%%%%%%%%%%%%%%%%%%%%%%%%%%%%%%%%%%%%%%%%%%%%%%

\begin{subsection}{The construction of projective F-schemes}
Suppose that $R=R_0\oplus R_1\oplus ...$ is a (nonzero) graded ring. Let $\text{P}(R)$
be the set of all homogeneously fully invertible systems $\mathcal{S}$ on $R$ for
which $R_\mathcal{S}$ is a $\mathbb{Z}$-crossed product ring.
Given a homogeneous fraction $\begin{bf} a \end{bf}=(a_1,...,a_n)$,
set $D_+(\begin{bf} a \end{bf})=\{\mathcal{S}\in\text{P}(R)|\, \begin{bf} a
\end{bf}\in \mathcal{S} \}$. It is clear that
$D_+(1)=\text{P}(R)$ and $D_+(0)=\emptyset$.
Moreover, $D_+(\begin{bf}
  a \end{bf})\cap D_+(\begin{bf} b \end{bf})=D_+(\langle \begin{bf} a \end{bf},
  \begin{bf} b \end{bf}\rangle)$, for any two homogeneous fractions $\begin{bf} a \end{bf}$ and
$\begin{bf} b \end{bf}$   of $R$. So the subsets $D_+(\begin{bf} a \end{bf})$'s
(we only need the fractions $\begin{bf} a \end{bf}$ for which $R_{\begin{bf} a\end{bf}}$ is a
$\mathbb{Z}$-crossed product ring)
form a base for a topology on $\text{P}(R)$.  The following lemma is easy to check.
\begin{lemma}\label{fundamentalopensubsetsprojectie P}
Let   $\begin{bf} a \end{bf}=(a_1,...,a_n)$ and $\begin{bf} b \end{bf}
=(b_1,...,b_n)$ be two homogeneous fractions of $R$ such that    $R_{\begin{bf}
  b \end{bf}}$ is a $\mathbb{Z}$-crossed product ring. Then we have
$D_+(\begin{bf} b \end{bf})\subset D_+(\begin{bf} a \end{bf})$ iff  there is a graded homomorphism
$R_{\begin{bf} a \end{bf}}\to R_{\begin{bf} b \end{bf}}$ making the following diagram commutative
\begin{displaymath}
\xymatrix{ & R \ar[rd]  \ar[r] & R_{\begin{bf} a \end{bf}} \ar[d]  \\
           &   & R_{\begin{bf} b \end{bf}}}
 \end{displaymath}
\end{lemma}
Similar to the
affine case (and the construction of projective schemes), one can define a sheaf of rings
$\mathcal{O}_{\text{P}(R)}$ on $\text{P}(R)$ whose stalk at $\mathcal{S}$ is just $R_{(\mathcal{S})}$.
More precisely, setting  $\mathcal{O}_{\text{P}(R)}(D_+(\begin{bf} a \end{bf}))=
R_{(\begin{bf} a \end{bf})} $ gives rise to a sheaf of rings on  $\text{P}(R)$.
So we obtain a ringed space $(\text{P}(R),\mathcal{O}_{\text{P}(R)})$.
\begin{proposition}\label{projectiveF-scheme P}
The ringed space $(\text{P}(R),\mathcal{O}_{\text{P}(R)})$  is an
F-scheme. Furthermore, for any homogeneous fraction $\begin{bf}
  a \end{bf}$ of $R$ such that   $R_{\begin{bf}
  a \end{bf}}$ is a $\mathbb{Z}$-crossed product ring, there is
an isomorphism $(D_+(\begin{bf}
  a \end{bf}),\mathcal{O}_{\text{P}(R)}|_{D_+(\begin{bf}
  a \end{bf})})\cong \text{F}(R_{(\begin{bf}
  a \end{bf})})$
as ringed spaces.
\end{proposition}
\begin{proof}
Since the open subsets $D_+(\begin{bf}
  a \end{bf})$  ($\begin{bf}
  a \end{bf}$ as in the proposition) cover  $\text{P}(R)$, it is enough to prove
the second assertion. Note that for such a fraction $\begin{bf}
  a \end{bf}$, if a homogeneously fully invertible system $\mathcal{S}$ of $R$ contains
  $\begin{bf} a \end{bf}$ then  $R_{\mathcal{S}}$
  contains an invertible homogeneous element of degree one.
Using Proposition \ref{F.I.SinA_Sgraded P}, we can identify
$D_+(\begin{bf}
  a \end{bf})$  with the set of homogeneously fully invertible systems on $R_{\begin{bf}
  a \end{bf}}$.
Now, Proposition \ref{homogeneousof A_a P} implies that $D_+(\begin{bf}
  a \end{bf})$ can be identified with
$\text{F}(R_{(\begin{bf}
  a \end{bf})})$. It is easy to see that this establishes a homeomorphism from
$D_+(\begin{bf}
  a \end{bf})$ onto  $\text{F}(R_{(\begin{bf}
  a \end{bf})})$. The second part of Proposition \ref{homogeneousof A_a P}
implies that  this homeomorphism can be completed into an isomorphism of ringed spaces.
\end{proof}
The F-scheme $(\text{P}(R),\mathcal{O}_{\text{P}(R)})$ (or simply $\text{P}(R) $)
is called the \textbf{projective F-scheme} associated to $R$.  Given a ring $A$,
the \textbf{projectile $n$-space} over $A$ is defined to be
$\mathbb{P}^{n}_A:=\text{P}(A\langle x_0,x_1,...,x_n\rangle)$ where
the graded structure on the ring $A\langle x_0,x_1,...,x_n\rangle$ is given by
$\deg(a)=0$ and $\deg(x_i)=1$ for any nonzero $a\in A$, $i=0,1,...,n$.
Similarly, the   \textbf{commutative projective $n$-space} over $A$ is defined to be
$\mathbb{P}^{n}_{A,c}:=\text{P}(A[x_0,x_1,...,x_n])$. The following is easy to prove
(use Proposition \ref{projectiveF-scheme P} and a similar result in scheme theory).
\begin{proposition}
Suppose that $A=A_0\oplus A_1\oplus ...$ is a commutative graded ring such that $A$
is generated by $A_1$ as an $A_0$-algebra. Then there is a canonical isomorphism
  $$\text{L}(\text{P}(A))
  \cong \text{Proj}\,A$$
 as schemes.
\end{proposition}
Now suppose that $R=R_0\oplus R_1\oplus ...$ and $R=T_0\oplus T_1\oplus ...$
are two graded rings. Let $\phi:R\to T$ be a graded homomorphism. Set $U$ to be the
set of all $\mathcal{S}\in \text{P}(T)$ for which $h(\phi^*(\mathcal{S}))\in
\text{P}(R)$. I claim that $U$ is an open subset of  $\text{P}(T)$. Suppose that
$\mathcal{S}\in U$. Then there is a homogenous fraction $\begin{bf} a\end{bf}$
of $R$ such that $\begin{bf} a\end{bf}\in h(\phi^*(\mathcal{S}))$
and $R_{\begin{bf} a\end{bf}}$ has an invertible homogeneous element of degree one.
Clearly $\phi(\begin{bf} a\end{bf})$ is a homogeneous fraction of $T$ and
it is  easy to see that $\mathcal{S}\in D_{+}(\phi(\begin{bf} a\end{bf}))\subset U$.
So $U$ is an open subset of $\text{P}(T)$. Define the function
$f:U\to \text{P}(R)$ by $f(\mathcal{S})=h(\phi^*(\mathcal{S}))$. It is easy to see that
$f^{-1}(D_{+}( \begin{bf} a\end{bf}))=D_{+}(\phi(\begin{bf} a\end{bf})))$
for any homogeneous fraction $\begin{bf} a\end{bf}$
of $R$. So $f$ is continuous. Using the natural homomorphisms
$R_{\begin{bf} a\end{bf}}\to T_{\phi(\begin{bf} a\end{bf})}$,
we can extend $f$ to a morphism $f:U\to \text{P}(R)$ of F-schemes. Note that
for any homogeneous fraction $\begin{bf} a\end{bf}$
of $R$, the morphism
$$f:D_{+}(\phi(\begin{bf} a\end{bf})))\cong\text{F}(T_{(\phi(\begin{bf} a\end{bf}))})
\to  D_{+}(\begin{bf} a\end{bf})\cong\text{F}(R_{(\begin{bf} a\end{bf})})$$
is the morphism induced by the natural ring homomorphism $R_{(\begin{bf} a\end{bf})}\to
T_{(\phi(\begin{bf} a\end{bf}))}$.
\begin{proposition}\label{closedimmersionsprojective P}
Suppose that the homomorphism $\phi:R\to T$ is onto. Then
$U=\text{P}(T)$ and the morphism $f:\text{P}(T)\to\text{P}(R)$ is a closed
immersion.
\end{proposition}
\begin{proof}
Since $\phi:R\to T$ is onto, the homomorphism
$R_{h(\phi^*(\mathcal{S}))}\to T_{\mathcal{S}}$
is onto for any
homogeneously fully invertible system $\mathcal{S}$ of $T$. Since this homomorphism is
homogeneously local, if $T_{\mathcal{S}}$ contains
an invertible homogeneous element of degree one then so does $R_{h(\phi^*(\mathcal{S}))}$.
This proves that $U= \text{P}(T)$. To prove that $f$ is a closed immersion,
it is enough to show that for any homogeneous fraction  $\begin{bf} a\end{bf}$
of $R$, the morphism
$$f:D_{+}(\phi(\begin{bf} a\end{bf})))\cong\text{F}(T_{(\phi(\begin{bf} a\end{bf}))})
\to  D_{+}(\begin{bf} a\end{bf})\cong\text{F}(R_{(\begin{bf} a\end{bf})})$$
is a closed immersion, see Corollary \ref{closedimmersionopenconver C}.
But this homomorphism is induced by the natural homomorphism  $R_{(\begin{bf} a\end{bf})}\to
T_{(\phi(\begin{bf} a\end{bf}))}$ which is onto. So $f$ is a closed immersion.
\end{proof}

\end{subsection}
%%%%%%%%%%%%%%%%%%%%%%%%%%%%%%%%%%%%%%%%%%%%%%%%%%%%%%%%%%%%%%%%%%%%%%%%%%%%%%%%%%%%%%%%%%%%%%%%
%%%%%%%%%%%%%%%%%%%%%%%%%%%%%%%%%%%%%%%%%%%%%%%%%%%%%%%%%%%%%%%%%%%%%%%%%%%%%%%%%%%%%%%%%%%%%%%%
%%%%%%%%%%%%%%%%%%%%%%%%%%%%%%%%%%%%%%%%%%%%%%%%%%%%%%%%%%%%%%%%%%%%%%%%%%%%%%%%%%%%%%%%%%%%%%%%
%%%%%%%%%%%%%%%%%%%%%%%%%%%%%%%%%%%%%%%%%%%%%%%%%%%%%%%%%%%%%%%%%%%%%%%%%%%%%%%%%%%%%%%%%%%%%%%%
 
%%%%%%%%%%%%%%%%%%%%%%%%%%%%%%%%%%%%%%%%%%%%%%%%%%%%%%%%%%%%%%%%%%%%%%%%%%%%%%%%%%%%%%%%%%%%%%%%
%%%%%%%%%%%%%%%%%%%%%%%%%%%%%%%%%%%%%%%%%%%%%%%%%%%%%%%%%%%%%%%%%%%%%%%%%%%%%%%%%%%%%%%%%%%%%%%%
%%%%%%%%%%%%%%%%%%%%%%%%%%%%%%%%%%%%%%%%%%%%%%%%%%%%%%%%%%%%%%%%%%%%%%%%%%%%%%%%%%%%%%%%%%%%%%%%
%%%%%%%%%%%%%%%%%%%%%%%%%%%%%%%%%%%%%%%%%%%%%%%%%%%%%%%%%%%%%%%%%%%%%%%%%%%%%%%%%%%%%%%%%%%%%%%%
%%%%%%%%%%%%%%%%%%%%%%%%%%%%%%%%%%%%%%%%%%%%%%%%%%%%%%%%%%%%%%%%%%%%%%%%%%%%%%%%%%%%%%%%%%%%%%%%
\end{section}

\end{document}